\documentclass[11pt]{amsart}

\usepackage{latexsym}
\usepackage{amsmath}
\usepackage{amsfonts}
\usepackage{amssymb}
\usepackage{amsthm}

\newcommand{\cD}{\mathcal{D}}
\newcommand{\K}{\mathcal{K}}
\renewcommand{\P}{\mathcal{P}}
\newcommand{\G}{\mathcal{G}}
\newcommand{\X}{\mathcal{X}}
\newcommand{\cY}{\mathcal{Y}}
\newcommand{\Y}{\mathcal{Y}}
\newcommand{\cZ}{\mathcal{Z}}

\newcommand{\gG}{\Gamma}

\newcommand{\T}{\mathbb{T}}
\newcommand{\Q}{\mathcal{Q}}

\newcommand{\R}{\mathbb{R}}
\newcommand{\E}{\mathbb{E}}
\newcommand{\N}{\mathbb{N}}

\newcommand{\Z}{\mathbb{Z}}

\newcommand{\norm}[1]{\left\Vert #1\right\Vert}
\newcommand{\nnorm}[1]{\lvert\!|\!| #1|\!|\!\rvert}
\theoremstyle{plain}
\newtheorem{theorem}{Theorem}[section]
\newtheorem{lemma}[theorem]{Lemma}
\newtheorem{proposition}[theorem]{Proposition}
\newtheorem*{proposition*}{Proposition}

\newtheorem*{claim*}{Claim}
\newtheorem*{lemma*}{Lemma}

\newtheorem*{theoremA'}{Theorem A'}
\newtheorem*{theoremB'}{Theorem B'}
\newtheorem*{theoremC'}{Theorem C'}

\newtheorem*{theorem*}{Theorem}

\newtheorem*{conjecture*}{Conjecture}
\newtheorem{corollary}[theorem]{Corollary}

\newtheorem*{vdCLemma}{Van der Corput's Lemma}

\theoremstyle{definition}

\newtheorem*{definition*}{Definition}
\newtheorem{example}{Example}

\theoremstyle{remark}

\newtheorem*{remark}{Remark}

\newtheorem*{notation}{Notation}

\newcommand{\CC}{\mathcal C}
\newcommand{\CD}{\mathcal D}
\newcommand{\CI}{\mathcal I}
\newcommand{\CK}{\mathcal K}

\newcommand{\CX}{\mathcal X}
\newcommand{\CZ}{\mathcal Z}

\newcommand{\QQ}{{\mathbb Q}}

\newcommand{\ux}{\underline x}
\newcommand{\uh}{{\underline h}}
\newcommand{\bI}{{\bf I}}

\DeclareMathOperator{\krat}{\CK_{\text{\rm rat}}}

\DeclareMathOperator{\vdc}{-vdC}
\DeclareMathOperator{\cl}{cl}

\begin{document}

\title[\tiny{Ergodic averages of commuting transformations with distinct degree...}]{Ergodic averages of commuting transformations with distinct degree polynomial iterates}


\author{Qing Chu}
\address[Qing Chu]{Universit\'e Paris-Est Marne-la-Vall\'ee, Laboratoire d'analyse et de math\'ematiques appliqu\'ees, UMR
CNRS 8050, 5 Bd Descartes, 77454 Marne la Vall\'ee Cedex, France} \email{qing.chu@univ-mlv.fr}

\author{Nikos Frantzikinakis}
\address[Nikos Frantzikinakis]{University of Crete, Department of mathematics, Knossos Avenue, Heraklion 71409, Greece} \email{frantzikinakis@gmail.com}

\author{Bernard Host}
\address[Bernard Host]{
Universit\'e Paris-Est Marne-la-Vall\'ee, Laboratoire d'analyse et de math\'ematiques appliqu\'ees, UMR
CNRS 8050, 5 Bd Descartes, 77454 Marne la Vall\'ee Cedex, France } \email{bernard.host@univ-mlv.fr}

\begin{abstract}
We prove mean convergence, as $N\to\infty$, for the multiple ergodic   averages $\frac{1}{N}\sum_{n=1}^N f_1(T_1^{p_1(n)}x)\cdot\ldots\cdot f_\ell(T_\ell^{p_\ell(n)}x)$, where $p_1,\ldots,p_\ell$ are integer polynomials with distinct degrees, and
$T_1,\ldots,T_\ell$ are commuting, invertible measure preserving
transformations, acting on the same
probability space. This establishes several cases of a conjecture of Bergelson and Leibman, that complement the case of linear polynomials, recently established by Tao.
Furthermore, we show that, unlike the case of linear polynomials, for polynomials of distinct degrees,
the corresponding characteristic factors are   mixtures of inverse limits of nilsystems. We use this particular structure, together with some equidistribution results on nilmanifolds, to give an application to multiple recurrence and a corresponding one to combinatorics.
\end{abstract}

\thanks{The  second author was partially supported by
 Marie Curie IRG  248008, and the third author was partially supported by the Institut Universitaire de France}



 \maketitle
\centerline{\today}
\setcounter{tocdepth}{1}
\tableofcontents

\section{Main results, ideas in the proofs, and further directions}
\subsection{Introduction and main results} A well studied and difficult problem in ergodic theory is the analysis of the limiting behavior of multiple ergodic averages of commuting transformations taken along polynomial iterates.
   A related conjecture of Bergelson and Leibman
(given explicitly in \cite{Be96}) states the following:
\begin{conjecture*}
Let $(X,\X,\mu)$ be a probability  space, $T_1, \ldots, T_\ell \colon X\to X$ be commuting, invertible
  measure preserving transformations,  $f_1,\ldots,f_\ell\in L^\infty(\mu)$,
and  $p_1,\ldots,p_\ell\in \Z[t]$.

Then  the  limit
\begin{equation}\label{E:Multies}
\lim_{N\to \infty} \frac{1}{N}\sum_{n=1}^{N} f_1(T_1^{p_1(n)}x)\cdot \ldots\cdot f_\ell(T_\ell^{p_\ell(n)}x)
\end{equation}
 exists in $L^2(\mu)$.
\end{conjecture*}
 Special forms of the averages in \eqref{E:Multies}  were
introduced  and studied by Furstenberg~\cite{F},
Furstenberg and Katznelson~\cite{FuKa},
and Bergelson and Leibman~\cite{BL},
in a depth that was sufficient for them to establish the theorem of
 Szemer\'edi on arithmetic progressions and its multidimensional and polynomial extensions respectively.

Proving convergence of these averages turned out to be a harder problem.
 When all the transformations $T_1,\ldots,T_\ell$ are equal,
convergence was established after
a long series of intermediate results; the papers    \cite{F,CL1, CL3, CL2,
FW, R, HKa, Zi07}  dealt with the
important case of linear  polynomials, and  using the machinery
 introduced in \cite{HKa}, convergence for arbitrary polynomials was finally obtained in~\cite{HKb}
except for a few cases that were treated in~\cite{L}.
For general commuting transformations, progress has been scarcer.
When all the polynomials in \eqref{E:Multies} are linear,
after a series of partial results \cite{FuKa, CL1, Les93a, Zh96,FK} that were obtained using ergodic theory,
convergence was established
  in~\cite{T} using a finitary argument. Subsequently, motivated by
ideas from~\cite{T}, several other proofs of this ``linear''  result
were found  using non-standard analysis~\cite{To}, and then ergodic theory
\cite{A, H}. Proofs of convergence for general  polynomial iterates  have
been given only under very strong ergodicity assumptions \cite{Be, J}.
On the other hand,  very recently, in \cite{A1, A2}
techniques from~\cite{A} have been refined and extended, aiming to eventually handle the case of general polynomial iterates.
Despite such intense efforts,   for general commuting transformations,
apart from the case where all the polynomials are linear,   no other instance of the conjecture of
Bergelson and Leibman has been resolved. In this article, we are going to establish
this conjecture when the polynomial iterates have distinct degrees:

\begin{theorem}\label{T:Conv}
Let $(X,\X,\mu)$ be a probability  space, $T_1, \ldots, T_\ell \colon X\to X$ be commuting,
invertible  measure preserving transformations, and  $f_1,\ldots,f_\ell\in L^\infty(\mu)$.
 Suppose that the polynomials $p_1,\ldots,p_\ell\in \Z[t]$ have distinct degrees.

Then  the limit
\begin{equation}\label{E:Multies1}
\lim_{N-M\to \infty} \frac{1}{N-M}\sum_{n=M}^{N-1} f_1(T_1^{p_1(n)}x)\cdot \ldots\cdot f_\ell(T_\ell^{p_\ell(n)}x)
\end{equation}
exists in $L^2(\mu)$.
\end{theorem}

 Unlike previous arguments in~\cite{ CL1, CL2, T, To, A, H}, where one finds ways to
 sidestep the problem of giving precise algebraic descriptions of the
 factor systems that control the limiting behavior of special cases of
 the averages \eqref{E:Multies1}, a distinctive feature of the proof
 of Theorem~\ref{T:Conv} is that we give such descriptions.\footnote{A
 key difference between the averages of
 $f_1(T_1^nx)\cdot f_2(T_2^nx)$ and the averages of
 $f_1(T_1^nx)\cdot f_2(T_2^{n^2}x)$ is
 that when $T_1=T_2$ the first one becomes ``degenerate''
 ($=$ averages of $(f_1\cdot f_2)(T_1^nx)$), and this
 complicates the structure of the possible factors that control their
 limiting behavior.
 However, no such choice of
 $T_1,T_2$ makes the second average ``degenerate''.}
 Furthermore, we
 did not find it advantageous to work within a suitable extension of
 our system in order to simplify our study (like the ``pleasant'' or ``magic'' extensions that were
 introduced in \cite{A} and in \cite{H} respectively).
 In this respect, our analysis is more closely related to the one used
 to study convergence results when all the transformations
 $T_1,\ldots,T_\ell$ are equal, and in fact  uses this well
 developed single transformation theory in a crucial way (in some special cases our approach leads to very concise proofs, see Appendix~\ref{sec:appendix}).
 The next result gives the description of the
 aforementioned factors (the factors $\cZ_{k,T_i}$ are defined in
 Section~\ref{SS:Zk}):

\begin{theorem}\label{T:CharA}
Let $(X,\X,\mu)$ be a probability  space, $T_1, \ldots, T_\ell \colon X\to X$ be commuting, invertible  measure preserving transformations, and  $f_1,\ldots,f_\ell\in L^\infty(\mu)$. Let  $p_1,\ldots,p_\ell\in \Z[t]$ be polynomials with distinct degrees and maximum degree $d$.

Then there  exists $k=k(d,\ell)\in \N$ such that: If
   $f_i\bot \cZ_{k,T_i}$ for some $i\in \{1,\ldots,\ell\}$, then
$$
\lim_{N-M\to \infty} \frac{1}{N-M}\sum_{n=M}^{N-1} f_1(T_1^{p_1(n)}x)\cdot \ldots\cdot f_\ell(T_\ell^{p_\ell(n)}x)=0
$$
in   $L^2(\mu)$.
\end{theorem}
Factors that satisfy the aforementioned convergence property are often  called \emph{characteristic factors}.
The utility of the characteristic factors obtained in Theorem~\ref{T:CharA}
stems from the fact that each individual
factor is a mixture of systems of algebraic origin, in particular, it is a mixture of  inverse limits of nilsystems \cite{HKa} (see also Theorem~\ref{T:Structure}).
Using this algebraic description  of the characteristic factors (in fact its consequence Proposition~\ref{L:ApprNil} is more suitable for our needs), and some equidistribution
results on nilmanifolds, we give the following application to multiple recurrence:
\begin{theorem}\label{T:LowerBounds}
Let $(X,\X,\mu)$ be a probability  space and $T_1, \ldots, T_\ell \colon X\to X$ be commuting, invertible
 measure preserving transformations.

Then for every choice of distinct positive integers $d_1,\ldots, d_\ell$, and every $\varepsilon>0$, the set
\begin{equation}\label{E:lowerbounds}
\{ n\in\N\colon \mu(A\cap T_1^{-n^{d_1}}A\cap \cdots\cap T_\ell^{-n^{d_\ell}}A)\geq \mu(A)^{\ell+1}-\varepsilon\}
\end{equation}
has bounded gaps.
\end{theorem}
If the integers are not distinct, say $d_1=d_2$, then the result fails.
For example,  one can take $T_2=T_1^2$, and choose  the (non-ergodic) transformation
$T_1$, and the set $A$,  so that \eqref{E:lowerbounds} fails  with any power of $\mu(A)$ on the right hand side
for every $n\in \N$
 (Theorem~2.1 in \cite{BHK}).
If $\ell=2$, $d_1=d_2=1$, and the joint action of the transformations $T_1,T_2$ is ergodic,
then the result remains true up to a change of the exponent on the right hand side \cite{Chu}.
 But even under similar ergodicity assumptions, the result probably fails when $3$ exponents agree no matter
 what exponent one uses on the right hand side
 (a conditional counterexample appears in Proposition~5.2 of \cite{Fr08}).

It will be clear from our argument that in  the statement of Theorem~\ref{T:LowerBounds}
we can
replace the polynomials $n^{d_1},\ldots,n^{d_\ell}$ by any  collection of polynomials $p_1,\ldots,p_\ell\in \Z[t]$
with zero constant terms that satisfy
 $t^{\text{deg}(p_i)+1}|p_{i+1}$  for $i=1,\ldots,\ell-1$. For example $\{n,n^3+n^2,n^5+n^4\}$ is such a family. On the other hand, our argument does not work for all polynomials with distinct degrees (the problem is to find a replacement for Lemma~\ref{L:NilEqui}), but the same lower bounds are expected to hold for any collection
of rational independent integer polynomials with zero constant terms.

 Using a multidimensional version of Furstenberg's
correspondence principle (see~\cite{FuKa} or~\cite{BL}) it is straightforward
to   give
 a combinatorial consequence of this result.
 We leave the routine details of the proof to the interested reader.
\begin{theorem}
 Let $k,\ell\in \N$,  $\Lambda\subset \Z^k$ with $\bar{d}(\Lambda)>0$,\footnote{For
a set  $\Lambda\subset\mathbb{\Z}^k$, we define its upper density by
$\bar{d}(\Lambda)=\limsup_{N\to\infty}|\Lambda\cap[-N,N]^k|/(2N)^k$.}
 and $v_1,\ldots,v_\ell$ be vectors in $\Z^k$.

 Then for every choice of distinct positive integers $d_1,\ldots, d_\ell$, and every $\varepsilon>0$, the set
\begin{equation}\label{E:bounds}
\{n\in\N\colon \quad
\bar{d}\bigl(\Lambda\cap(\Lambda+n^{d_1}v_1)\cap\cdots\cap (\Lambda
+n^{d_\ell}v_{\ell})\bigr)\geq \bar{d}(\Lambda)^{\ell+1}-\varepsilon\}
\end{equation}
has bounded gaps.
\end{theorem}

\subsection{Ideas in the proofs of the main results}
\label{subsec:ideas}

\subsubsection{Key ingredients}

The proofs of Theorems~\ref{T:Conv}, \ref{T:CharA},  and \ref{T:LowerBounds}, use
several
ingredients.

\medskip \noindent \emph{Van der Corput's Lemma.} We are going to use    repeatedly
 the following  variation of the classical elementary
lemma of van der Corput.
Its proof is a straightforward modification of  the one
given in \cite{Be}.

\begin{vdCLemma}
Let  $(v_{n})$ be a bounded  sequence of vectors in a Hilbert space.
Let
$$
b_h=\overline{\lim}_{N-M\to \infty}\Big|\frac{1}{N-M}
\sum_{n=M}^{N-1}<v_{n+h},v_{n}>\Big|.
$$
Suppose that
$$
\lim_{H\to\infty}\frac{1}{H}\sum_{h=1}^H b_h=0.
$$
Then
$$
\lim_{N-M\to\infty}
\norm{\frac{1}{N-M}\sum_{n=M}^{N-1} v_{n}}=0.
$$
\end{vdCLemma}
\noindent In most applications we have $b_h=0$ for every sufficiently large
$h$, or for  ``almost every $h$'', meaning that the exceptional set
has zero upper density.

\medskip \noindent \emph{An approximation result.} This enables us in several instances to
 replace sequences of the form $(f(T^nx))_{n\in\N}$, where $f$ is a $\cZ_{k,T}$-measurable function,
with $k$-step nilsequences.
For ergodic systems,  this result is an easy consequence of the structure theorem for the factors $\cZ_{k,T}$ (Theorem~\ref{T:Structure}).
But we need a harder to establish  non-ergodic version (Proposition~\ref{L:ApprNil});
in our context we  cannot assume that each individual transformation is ergodic.

\medskip \noindent \emph{Nilsequence correlation estimates.} These, roughly speaking, assert that
 ``uniform''   sequences do not correlate with nilsequences (see for example Theorem~\ref{T:HKdirect}).

\medskip \noindent \emph{Equidistribution results on nilmanifolds.}
 These will only be used  in the proof
of Theorem~\ref{T:LowerBounds} (see Section~\ref{SS:equidistribution}).

\subsubsection{Combining the key ingredients}

We first prove Theorem~\ref{T:CharA} that   provides convenient characteristic factors for
the multiple ergodic averages   in \eqref{E:Multies1}. Its  proof proceeds in two steps: $(i)$ In Sections~\ref{S:Char2} and \ref{S:Char}
we use a PET-induction argument based
on successive uses of van der Corput's Lemma to find a characteristic factor for the transformation
that corresponds to the highest degree polynomial iterate, and
$(ii)$ In Section~\ref{S:Charlower} we combine step $(i)$,  with the aforementioned approximation result
and  nilsequence correlation estimates, to find characteristic factors for
the other transformations as well.

The strategy for proving Theorems~\ref{T:Conv} and \ref{T:LowerBounds}
can be summarized as follows: In order to study the limit \eqref{E:Multies1}, we first use Theorem~\ref{T:CharA} to reduce matters to the case where all the functions $f_i$ are $\cZ_{k,T_i}$-measurable for some $k\in \N$, and then the aforementioned  approximation result
to  reduce matters to   establishing certain convergence or equidistribution properties on nilmanifolds.
This last step is  easy to carry out when
proving Theorem~\ref{T:Conv} (see Section~\ref{SS:ProofA}), but becomes much more cumbersome when  proving
 Theorem~\ref{T:LowerBounds}. We prove the equidistribution properties needed for  Theorem~\ref{T:LowerBounds} in Section~\ref{S:lower}.

\subsection{Further directions}
When $\ell=2$ and $p_1(n)=p_2(n)=n$, it is known that
some sort of commutativity assumption on the transformations $T_1,T_2$  has to be made in order for the limit \eqref{E:Multies1} to exist   in $L^2(\mu)$ (see \cite{BL04} for  examples where convergence fails
when $T_1$ and $T_2$ generate   solvable groups of exponential growth). On the other hand, it is not clear whether a similar assumption
is necessary
when say $\ell=2$ and $p_1(n)=n$, $p_2(n)=n^2$.
In fact, it could be the case that for  Theorems~\ref{T:Conv}, \ref{T:CharA}, and  \ref{T:LowerBounds}, no commutativity assumption at all is needed.

Since convergence of the averages in \eqref{E:Multies1} for $\CZ_k$-measurable functions can be shown for general families of integer  polynomials (see the argument in Section~\ref{SS:proof}), it follows that
the  averages in \eqref{E:Multies1} converge in $L^2(\mu)$ for any collection of polynomials for which
the conclusion of Theorem~\ref{T:CharA} holds. We conjecture that  the conclusion of Theorem~\ref{T:CharA}
holds if and only if the  family of polynomials $p_1,\ldots,p_\ell$ is  pairwise independent,
meaning, the set $\{1,p_i,p_j\}$ is linearly independent for every $i,j\in \{1,\ldots,\ell\}$ with $i\neq j$
(simple examples show that the condition is necessary).
Furthermore, we conjecture that if the polynomials $1,p_1,\ldots,p_\ell$ are linearly independent, then  the factors $\krat(T_i)$ can take the place of the factors $\cZ_{k,T_i}$ in the hypothesis of Theorem~\ref{T:CharA}.

In the case where all the transformations $T_1,\ldots,T_\ell$ are equal, the conclusion of
Theorem~\ref{T:LowerBounds} is known to hold whenever the polynomials
$n,n^2,\ldots,n^\ell$ are replaced by  any  family  of linearly independent polynomials $p_1,\ldots,p_\ell$,
each having zero constant term \cite{FrK06} (it is known that this independence assumption
is  necessary \cite{BHK}).
We conjecture that a similar result holds for any family of commuting, invertible measure preserving transformations $T_1,\ldots,T_\ell$. And in fact again, the assumption that the transformations $T_1,\ldots,T_\ell$ commute
may be superfluous.

In most cases where the family of polynomials $p_1,\ldots,p_\ell$ is not pairwise independent, for example when $p_1(n)=\ldots=p_\ell(n)=n^2$,  the methods of the present article do not suffice to
study  the limiting behavior of the averages \eqref{E:Multies1}.\footnote{There are particular (but rather exceptional) cases of non-pairwise independent families of polynomials, where  the methods of the present paper can be easily modified and combined with the known ``linear'' results to prove convergence. One such example is when $p_1(n)=n,\ldots,p_{\ell-1}(n)=n$, and $p_\ell(n)$ is a  polynomial with a sufficiently large degree (degree $>2^{\ell}$ makes the problem accessible to the ``simple'' methods of the Appendix).}  It is in cases like this that
working with some kind of ``pleasant" extension (using terminology from \cite{A}) or ``magic'' extension (using terminology form \cite{H})  of the system may offer an essential advantage (this is indeed the case when all the polynomials
are linear).

\subsection{General conventions and notation}
By a \emph{system} we mean a Lebesgue probability space $(X,\CX,\mu)$,
endowed with a single, or several commuting, invertible  measure preserving
 transformations, acting on $X$.

For notational convenience, all functions are implicitly assumed to
be real valued, but straightforward modifications of
our arguments, definitions, etc, can be given
 for complex valued functions as well.

We say  that \emph{the averages of the sequence $(a_n)_{n\in\N}$
converge} to some limit $L$, and we write
$$
 \lim_{N-M\to+\infty}\frac 1{N-M}\sum_{n=M}^{N-1}a_n=L,
$$
if the averages of $a_n$ on any sequence of intervals whose lengths
tend to infinity converge to $L$.
We use similar formulations for  the $\limsup$ and for limits in function
spaces.

Lastly, the following notation will be
used throughout the article: $\N=\{1,2,\ldots\}$, $\T^k=\R^k/\Z^k$, $Tf=f\circ T$,
$e(t)=e^{2\pi i t}$.

\section{Background in ergodic theory and nilmanifolds}

\subsection{Background in ergodic theory}
Let  $(X,\X,\mu, T)$ be a system.

\medskip
\noindent \emph{Factors.}
A {\it homomorphism} from  $(X,\X,\mu, T)$ to a
system $(Y, \cY, \nu, S)$ is a measurable map $\pi\colon X'\to Y'$,
where $X'$ is a $T$-invariant subset of $X$ and $Y'$ is an
$S$-invariant subset of $Y$, both of full measure, such that
$\mu\circ\pi^{-1} = \nu$ and $S\circ\pi(x) = \pi\circ T(x)$ for
 $x\in X'$. When we have such a homomorphism we say that the system $(Y, \Y,
\nu, S)$ is a {\it factor} of the system $(X,\X,\mu, T)$.  If
the factor map $\pi\colon X'\to Y'$ can be chosen to be  bijective,
then we say that the systems $(X,\X, \mu, T)$ and $(Y, \Y, \nu, S)$
are {\it isomorphic} (bijective maps on Lebesgue spaces have
measurable inverses).

A factor can be characterized (modulo isomorphism) by
$\pi^{-1}(\cY)$, which is a $T$-invariant sub-$\sigma$-algebra
of $\mathcal B$, and conversely any $T$-invariant sub-$\sigma$-algebra of
$\mathcal B$ defines a factor. By a classical abuse of terminology we
denote by the same letter the $\sigma$-algebra $\Y$ and its
inverse image by $\pi$. In other words, if $(Y, \Y, \nu, S)$ is a
factor of $(X,\X,\mu,T)$, we think of $\Y$ as a
sub-$\sigma$-algebra of $\X$. A factor can also be
characterized (modulo isomorphism) by a $T$-invariant subalgebra
$\mathcal{F}$ of $L^\infty(X,\X,\mu)$, in which case $\Y$ is
the sub-$\sigma$-algebra generated by $\mathcal{F}$, or equivalently,
$L^2(X,\Y,\mu)$ is the closure of $\mathcal{F}$ in
$L^2(X,\X,\mu)$.

\medskip
\noindent \emph{Inverse limits.}
We say that $(X,\X,\mu,T)$ is an {\it inverse limit of a sequence of
  factors} $(X,\X_j,\mu,T)$ if $(\X_j)_{j\in\mathbb{N}}$ is an
increasing sequence of $T$-invariant sub-$\sigma$-algebras such that
$\bigvee_{j\in\N}\X_j=\X$ up to sets of measure
zero.

\medskip
\noindent \emph{Conditional expectation.}
If $\Y$ is a $T$-invariant sub-$\sigma$-algebra of $\X$ and $f\in
L^1(\mu)$, we write $\E(f|\Y)$, or $\E_\mu(f|\Y)$ if needed, for the
 conditional expectation of $f$ with respect to $\Y$.
 We will frequently make use
of the identities
$$
\int \E(f|\Y) \ d\mu= \int f\ d\mu\text{ and }
T\,\E(f|\Y)=\E(Tf|\Y).
$$
 We say that a function $f$ is orthogonal to $\Y$, and we write
$f\perp\Y$, when it has a zero conditional expectation on $\Y$.
If a function $f\in L^\infty(\mu)$ is measurable with respect to the factor $\Y$,
we write $f\in L^\infty(\Y,\mu)$.

\medskip
\noindent \emph{Ergodic decomposition.}
We write $\CI$, or $\CI(T)$ if needed, for the
$\sigma$-algebra  $\{A\in \X\colon T^{-1}A=A\}$ of invariant sets.
A system is \emph{ergodic} if all the  $T$-invariant sets have measure
either $0$ or $1$.

Let $x\mapsto\mu_x$ be a regular version of the
conditional measures with respect to the $\sigma$-algebra $\CI$.
This means that the map $x\mapsto\mu_x$ is
$\CI$-measurable, and
for very bounded measurable function $f$ we have
$$
 \E_\mu(f| \CI)(x)=\int f\,d\mu_x\ \text{for $\mu$-almost every }x\in X.
$$
Then the \emph{ergodic decomposition} of $\mu$ is
$$
 \mu=\int\mu_x\,d\mu(x).
$$
The measures $\mu_x$ have the additional property that for $\mu$-almost every $x\in X$ the system $(X,\X,\mu_x,T)$ is ergodic.

\medskip
\noindent \emph{The rational Kronecker factor.}
 For every $d\in\N$ we define $\CK_d=\CI(T^d)$. The rational Kronecker factor
is
$$
 \krat=\bigvee_{d\geq 1}\CK_d.
$$
We write $\krat(T)$, or  $\krat(T,\mu)$,  when needed.
This factor  is  spanned by the family of functions
$$
\{f\in L^\infty(\mu)\colon T^df=f \text{ for some } d\in\N\},
$$
 or, equivalently, by the family
$$
 \{f\in L^\infty(\mu)\colon Tf= e(a)\cdot f \text{ for some }
a\in\mathbb{Q}\}.
$$
 If $\E_\mu(f_1|\krat(T,\mu))=0$, then we have,  for $\mu$-almost every
$x\in X$, that  $\E_{\mu_x}(f_1|\krat(T,\mu_x))=0$ (see Lemma~3.2 in \cite{FrK06}).

\subsection{The seminorms $\nnorm{\cdot}_k$ and the factors $\cZ_k$}
\label{SS:Zk}
 Sections 3 and 4 of~\cite{HKa} contain constructions that associate to every
ergodic system a sequence of measures, seminorms, and factors. It
is the case that for  these constructions the hypothesis of ergodicity is not
needed.  Most properties remain valid, and can be proved in exactly the same manner,
 for general, not necessarily ergodic systems. We review the definitions and results we
need in the sequel.

Let $(X,\mathcal{X},\mu,T)$ be a  system.
We write $\mu=\int \mu_x \ d\mu(x)$ for the ergodic decomposition of
$\mu$.
\medskip

\noindent \emph{Definition of the seminorms.} For every $k\geq 1$,
we define  a measure $\mu^{[k]}$ on $X^{2^k}$ invariant under
$T\times T\times\cdots\times T$ ($2^k$ times), by
\begin{gather*}
\mu^{[1]}=\mu\times_{\CI(T)}\mu=\int\mu_x\times\mu_x\,d\mu(x)\ ; \\
\text{ for }k\geq 1,\ \mu^{[k+1]}=\mu^{[k]}\times_{\CI(T\times
T\times\cdots\times T)} \mu^{[k]}.
\end{gather*}

Writing $\ux=(x_0,x_1,\cdots,x_{2^k-1})$ for a point of $X^{2^k}$,
we
 define  a seminorm $\nnorm\cdot_ k$ on $L^\infty(\mu)$ by
$$
 \nnorm
f_k=\Bigl(\int\prod_{i=0}^{2^k-1}f(x_i)\,d\mu^{[k]}(\ux)
\Bigr)^{1/2^k}.
$$
That $\nnorm{\cdot}_k$ is a seminorm can be proved as in \cite{HKa}, and also follows
from the estimate \eqref{E:CSGHK} below.
If needed, we are going to write  $\nnorm{\cdot}_{k,\mu}$, or $\nnorm{\cdot}_{k,T}$.

By the inductive definition of the measures $\mu^{[k]}$ we have
\begin{gather}
\label{E:first}
 \nnorm{f}_{1}=\norm{\E(f|\mathcal{I})}_{L^2(\mu)}\ ;\\
\label{E:recur}
\nnorm f_{k+1}^{2^{k+1}} =\lim_{N-M\to\infty}\frac{1}{N-M}\sum_{n=M}^{N-1}
\nnorm{f\cdot T^nf}_{k}^{2^{k}}.
\end{gather}
This can be considered  as an alternate definition of the
seminorms (assuming one first establishes existence of the limit in \eqref{E:recur}).

For functions  $f_0,f_1,\ldots,f_{2^k-1}\in L^\infty(\mu)$, the next inequality (\cite{HKa}, Lemma 3.9) follows from the definition of the measures by a repeated use of the Cauchy-Schwarz inequality
\begin{equation}
\label{E:CSGHK}
\Bigl|\int\prod_{i=0}^{2^k-1}f_i(x_i)\,d\mu^{[k]}(\ux)\Bigr|\leq\prod_{i=0}^{2^k-1}
\nnorm{f_i}_{k}.
\end{equation}

\subsubsection*{Seminorms and ergodic decomposition}
By induction, for every $k\in\N$  we have
\begin{equation}
\label{eq:decomp_muk}
\mu^{[k]}=\int(\mu_x)^{[k]}\,d\mu(x).
\end{equation}
Therefore,  for every  function $f\in L^\infty(\mu)$ we have
\begin{equation}\label{E:nonergodic}
\nnorm{f}_{k,\mu}^{2^k}=\int \nnorm{f}_{k,\mu_x}^{2^k} \ d\mu(x).
\end{equation}

\subsubsection*{The factors $\CZ_k$}

For every $k\geq 1$, an invariant $\sigma$-algebra $\CZ_k$ on $X$ is
constructed exactly as in Section~4 of \cite{HKa}. It satisfies the same property as in
Lemma~4.3 of \cite{HKa}
\begin{equation}
\label{E:DefZ_l}
\text{\em for } f\in L^\infty(\mu),\ \E_\mu(f|\cZ_{k-1})=0\text{ \em if and
  only if }\  \nnorm f_{k,\mu} = 0.
\end{equation}
Equivalently, one has
\begin{equation}\label{E:DefZ_l'}
L^\infty(\cZ_{k-1},\mu)
=\Big\{f\in L^\infty(\mu)\colon \int f\cdot g \ d\mu=0
 \text{ for every } g\in L^\infty(\mu) \text{ with }
\nnorm g_{k} = 0\Big\}.
\end{equation}
In particular, if $f\in L^\infty(\mu)$ is measurable with respect to
$\cZ_{k-1}$ and satisfies $\nnorm f_k=0$, then $f=0$. Therefore,
 $$
\nnorm\cdot_k\text{ is a norm on } L^\infty(\cZ_{k-1},\mu).
$$
If further clarification is needed, we are going to write  $\cZ_{k,\mu}$, or $\cZ_{k,T}$.
If $f\in L^\infty(\mu)$, then it follows from  \eqref{E:nonergodic} and \eqref{E:DefZ_l} that
\begin{equation}
\label{eq:ZmuZmux}
 \E_\mu(f|\mathcal{Z}_{k,\mu})=0 \text{ if and only if }
\E_{\mu_x}(f|\mathcal{Z}_{k,\mu_x})=0 \text{ for }
 \mu\text{-almost every }  x\in X.
\end{equation}
Furthermore, if $f\in L^\infty(\mu)$, then                                       %
  $$
  f\in L^\infty(\cZ_{k,\mu},\mu) \text{ if and only if } f\in L^\infty(\cZ_{k,\mu_x},\mu_x) \text{ for }        %
   \mu\text{-almost every } x\in X.
   $$                                                                     %
   The first implication is non-trivial to establish though, due to various measurability problems.
We prove this in Corollary~\ref{cor:Zkmux} below.

For every $\ell\in \N$ one has $\nnorm{f}_{1,T^\ell}\ll_\ell \nnorm{f}_{2,T}$ (see proof of Proposition 2 in \cite{HKb}).
Using this and the inductive definition of the seminorms \eqref{E:recur},
one sees that  $\nnorm{f}_{k,T^\ell}\ll_{k,\ell}\nnorm{f}_{k+1,T}$.
Therefore,
\begin{equation}\label{E:powers}
\text{if} \quad f\bot L^\infty(\cZ_{k,T},\mu) \ \text{ then } \
 f\bot L^\infty(\cZ_{k-1,T^\ell},\mu) \ \text{ for every } \ell\in \N.
\end{equation}
\subsection{Nilsystems, nilsequences,  the structure of  $\cZ_k$}
A \emph{nilmanifold}  is a homogeneous space $X=G/\Gamma$ where  $G$ is a nilpotent Lie group,
and $\Gamma$ is a discrete cocompact subgroup of $G$. If $G_{k+1}=\{e\}$ , where $G_k$ denotes the $k$-the commutator subgroup of $G$, we say that $X$ is a
 $k$-\emph{step nilmanifold}.

 A $k$-step nilpotent Lie group $G$ acts on $G/\gG$ by left
translation where the translation by a fixed element $a\in G$ is given
by $T_{a}(g\gG) = (ag) \gG$.  By $m_X$ we denote the unique probability
measure on $X$ that is invariant under the action of $G$ by left
translations (called the {\it normalized Haar measure}), and by $\G/\gG$ we denote the
Borel $\sigma$-algebra of $G/\gG$. Fixing an element $a\in G$, we call
the system $(G/\gG, \G/\gG, m_X, T_{a})$ a {\it $k$-step  nilsystem}.

If $X=G/\Gamma$ is a $k$-step nilmanifold,  $a\in G$,  $x\in X$, and
$f\in \CC(X)$,
 we call the sequence $(f(a^nx))_{n\in\N}$ a \emph{basic $k$-step nilsequence}.
A \emph{$k$-step nilsequence}, is a uniform limit of \emph{basic $k$-step nilsequences}. As is easily verified,
the collection of $k$-step nilsequences, with the topology of uniform convergence, forms a closed algebra.
 We caution the reader that in other articles the term $k$-step nilsequence is used for what we call here
 basic $k$-step nilsequence, and in some instances the function $f$ is assumed to satisfy weaker or stronger
 conditions than continuity.

The connection between the factors $\cZ_k$ of a given ergodic system and
nilsystems is given by the following structure theorem (\cite{HKa}, Lemma 4.3, Definition 4.10,
and Theorem 10.1):
\begin{theorem}[{\bf\cite{HKa}}]\label{T:Structure}
  Let $(X,\X,\mu,T)$ be an ergodic  system and $k\in\N$.

  Then  the  system  $(X,\mathcal{Z}_{k},\mu,T)$ is a (measure theoretic)  inverse limit of  $k$-step
  nilsystems.
\end{theorem}
\begin{remark} In fact, in \cite{HKM} it is  shown that, for ergodic systems,  the factor $(X,\mathcal{Z}_{k},\mu,T)$ is (measurably) isomorphic to a topological
inverse limit of ergodic $k$-step nilsystems (for a definition see \cite{HKM}). We are going to use this fact later.\footnote{A \emph{topological dynamical system} is a pair $(Y,S)$ where $Y$ is a compact metric space and $S\colon Y\to Y$ is a continuous transformation. If $(Y_i,S_i)_{i\in\N}$ is a sequence topological dynamical systems and $\pi_i\colon Y_{i+1}\to Y_i$ are factor maps, the \emph{inverse limit} of the systems is defined to be the compact subset $Y$ of  $\prod_{i\in\N}Y_i$ given by $Y=\{(y_i)_{i\in\N}\colon \pi_i(y_{i+1})=y_i\}$, with the induced infinite product metric and continuous transformation $T$.}
\end{remark}

\subsection{Characteristic factors for linear averages}
Using successive applications of van der Corput's lemma, the following can be proved by  induction on $\ell$ as in Theorem~12.1 of \cite{HKa} (the $\ell=2$ case follows for example  from Theorem~2.1 in \cite{FW}):
\begin{theorem}\label{T:linear}
Let $\ell\geq 2$, $(X,\X,\mu, T)$ be a system, $f_1,\ldots,f_\ell\in L^\infty(\mu)$,
 and  $a_1,\ldots,a_\ell$ be  distinct non-zero integers.
 Suppose that $f_i\bot \cZ_{\ell-1}$ for some $i\in \{1,\ldots,\ell\}$.

Then  the averages
$$
\frac{1}{N-M}\sum_{N=M}^{N-1} f_1(T^{a_1n}x)\cdot \ldots \cdot f_\ell(T^{a_\ell n}x)
$$
converge to $0$ in $L^2(\mu)$.
\end{theorem}

\section{A key approximation property}\label{S:ApprNil}
In this section we are going to prove the following  key approximation result:
\begin{proposition}\label{L:ApprNil}
Let $(X,\X,\mu,T)$ be a system (not necessarily ergodic) and suppose that
$f\in L^{\infty}(\cZ_k,\mu)$ for some $k\in \N$.

Then for every $\varepsilon>0$ there exists a function
$\tilde{f}\in L^\infty(\mu)$, with $L^\infty$-norm bounded by $\norm{f}_{L^\infty(\mu)}$, such that
\begin{enumerate}
\item\label{it:tildeapprox}
 $\tilde{f}\in L^{\infty}(\cZ_k,\mu)$ and
$\norm{f-\tilde{f}}_{L^1(\mu)}\leq \varepsilon$\ ;
\item\label{it:tildenil}
 for $\mu$-almost every $x\in X$, the sequence $(\tilde f (T^nx))_{n\in\N}$
is a $k$-step nilsequence.
\end{enumerate}
\end{proposition}

By \cite{L}, if $(a_n)_{n\in\N}$ is a $k$-step nilsequence, and $p\in\Z[t]$
is a polynomial of degree $d\geq 1$, then the sequence
$(a_{p(n)})_{n\in\N}$ is a $(dk)$-step nilsequence. Therefore, the
function $\tilde f$ given by the proposition satisfies:
\begin{enumerate}
\setcounter{enumi}{2}
\item\label{it:tildepolnil}
\em
for $\mu$-almost every $x\in X$ and  every polynomial $p\in\Z[t]$ of
degree $d\geq 1$, the sequence $\bigl(\tilde
f(T^{p(n)}x)\bigr)_{n\in\N}$ is a $(dk)$-step nilsequence.
\end{enumerate}

If the system $(X,\X,\mu,T)$ is ergodic, then one can deduce  Proposition~\ref{L:ApprNil} from Theorem~\ref{T:Structure}
in a straightforward way. It turns out to be much harder to prove
this result
in the non-ergodic case (and this strengthening  is crucial for our later
applications),
due to a non-trivial measurable selection problem that one has to overcome.
We give the proof in the following subsections.

\subsection{Dual functions}\label{subsec:dual}
In this subsection, $(X,\X,\mu,T)$ is a system, and  the ergodic
decomposition of $\mu$ is $\mu=\int\mu_x\,d\mu(x)$.
We  remind the reader that we work with real valued functions
only.

We  define a family of functions that will be used in the proof of Proposition~\ref{L:ApprNil} and gather some basic properties they satisfy.

When $f$ is a bounded measurable function on $X$, for every $N\in\N$, we write
$$
A_N(f)=\frac{1}{N^k}
\sum_{1\leq n_1,\ldots,n_k\leq N} \prod_{\substack{\epsilon\in \{0,1\}^k,\\
\epsilon \neq 00\cdots 0}}
 f(T^{n_1\epsilon_1+\cdots +n_k\epsilon_k}x).
$$
 It is known  by  Theorem 1.2 in~\cite{HKa} that the averages $A_N(f)$ converge
in $L^2(\mu)$
 (in fact by \cite{As} they converge pointwise but we do not need this strengthening), and we define
$$
  \cD_k f=\lim_{N\to\infty}A_N(f)
$$
where the limit is taken in $L^2(\mu)$. If needed,  we write $\cD_{k,\mu}f$. The function $\cD_k f$
satisfies (\cite{HKa}, Theorem 13.1): For every $g\in L^\infty(\mu)$, we have
\begin{equation}
\label{eq:CDfg}
\int g\cdot\CD_kf\,d\mu=
\int g(x_0)\prod_{i=1}^{2^k-1}f(x_i)\,d\mu^{[k]}(\ux)
\end{equation}
where $\ux=(x_0,x_1,\cdots,x_{2^k-1})\in X^{2^k}$. In particular, by
the definition of $\nnorm f_k$, we have
\begin{equation}
\label{eq:CDff}
\int f\cdot\CD_kf\,d\mu=\nnorm f_k^{2^k},
\end{equation}
and by inequality~\eqref{E:CSGHK}, for every function $g\in
L^\infty(\mu)$ we have
\begin{equation}
\label{eq:boundcdfg}
\Bigl|\int g\cdot\CD_kf\,d\mu\Bigr|\leq\nnorm g_k\cdot\nnorm
f_k^{2^k-1}.
\end{equation}
\begin{example}
We have
$$
\cD_1f=\lim_{N\to\infty}\frac{1}{N} \sum_{n_1=1}^N T^nf=\E(f |\CI).
$$
 Also,
z If $T$ is an ergodic rotation on the circle $\T$ with the Haar
measure $m_\T$, then an easy computation gives
 $$
 (\cD_2f)(x)= \int_\T \int_\T f(x+s)\cdot f(x+t)\cdot f(x+s+t) \,
dm_\T(s) \, dm_\T(t).
 $$
 Notice that in this case the function $\cD_2f(x)$ may be  non-constant,
and no matter whether the function $f$ is continuous or not,   the function $\cD_2f(x)$ is  continuous
 on $\T$.
\end{example}

We gather some additional basic properties of dual functions.

\begin{proposition}\label{P:dual}
Let $(X,\X,\mu,T)$ be a system.

Then for every $f\in L^\infty(\mu)$ and  $k\in\N$ the following hold:
\begin{enumerate}
\item\label{it:dual1}
For $\mu$-almost every $x\in X$, we have
$\cD_{k,\mu}f=\cD_{k,\mu_x}f$ as functions of $L^\infty(\mu_x)$.

\item\label{it:dual4}
The function $\cD_kf$ is $\cZ_{k-1}$-measurable, in fact
$\cD_kf=\cD_k\tilde f$  where $\tilde f=\E(f|\cZ_{k-1})$.

\item\label{it:dual5}
 Linear combinations of functions $\cD_kf$ with $f\in L^\infty(\mu)$ are dense in
$L^1(\cZ_{k-1},\mu)$.
\end{enumerate}
\end{proposition}

\begin{proof}
We show~\eqref{it:dual1}. The averages $A_N(f)$ converge to $\cD_{k,\mu}f$ in
$L^2(\mu)$. Therefore,  there exists a subsequence of $A_N(f)$ that converges to
$\cD_{k,\mu}f$  almost everywhere with respect to $\mu$.
As a consequence,  for $\mu$-almost every $x\in X$,
this subsequence converges to $\cD_{k,\mu}f$
 almost everywhere with respect to $\mu_x$. On the other hand, by the definition of $\cD_{k,\mu_x}f$,
for $\mu$-almost every $x\in X$ this subsequence also converges to $\cD_{k,\mu_x}f$
in $L^2(\mu_x)$. The result follows.

We show~\eqref{it:dual4}. Since the operation $\cD_k$ maps $L^\infty(\cZ_{k-1},\mu)$
to itself, it suffices to establish the second claim.
Let $g\in L^\infty(\mu)$.
 Using~\eqref{eq:CDfg} and expanding $f$ as
$\tilde f+(f-\tilde f)$,
we see
that $\int g\cdot\CD_k f\,d\mu$  is equal to
$\int g\cdot\CD_k \tilde f\,d\mu$,
plus integrals  of the form
$$
 \int g(x_0)\cdot\prod_{i=1}^{2^k-1}f_i(x_i)\,d\mu^{[k]}(\ux),
$$
where each of the functions $f_i$ is equal to  either $\tilde f$ or to
$f-\tilde f$, and at least one of the functions $f_i$ is  equal to $f-\tilde f$.
Since $\E(f-\tilde f|\CZ_{k-1})=0$, by~\eqref{E:DefZ_l} we have
$\nnorm{f-\tilde f}_k=0$, and
by  inequality~\eqref{E:CSGHK},
all these integrals are equal to zero.
 This establishes that
$\int g\cdot\CD_k f\,d\mu$  is equal to
$\int  g\cdot\CD_k \tilde f\,d\mu$, and the
announced result follows.

We show~\eqref{it:dual5}.
By duality, it suffices to
show that if $g\in L^\infty(\cZ_{k-1},\mu)$ satisfies $\int g \cdot \cD_kf\ d\mu=0$ for every $f\in L^\infty(\mu)$, then $g=0$.
 Taking $f=g$ gives  $\int g \cdot \cD_kg\ d\mu=0$,  and
using~\eqref{eq:CDff} we get $\nnorm{g}_k=0$. Since
  $\nnorm{\cdot}_k$ is a norm in $L^\infty(\cZ_{k-1},\mu)$ we deduce that $g=0$. This completes the proof.
\end{proof}
\begin{corollary}
\label{cor:Zkmux}
Let $(X,\X,\mu,T)$ be a system and
 $f\in L^\infty(\CZ_{k,\mu},\mu)$ for some $k\in\N$.

  Then, for $\mu$-almost every $x\in X$, we have
$f\in  L^\infty(\CZ_{k,\mu_x},\mu_x)$.
\end{corollary}
\begin{proof}
By part~\eqref{it:dual5} of Propositition~\ref{P:dual}, there exists a
sequence $(f_n)_{n\in\N}$, of  finite linear combinations
of functions of the form $\CD_{k+1}\phi$ where $\phi\in
L^\infty(\mu)$, such that  $f_n\to f$ in
$L^1(\mu)$. Passing to a subsequence, we can assume that
$f_n\to f$  almost everywhere with respect to $\mu$. As a consequence, for $\mu$-almost
every $x\in X$, we have $f_n\to f$ almost everywhere with respect to $\mu_x$.

Furthermore, by parts~\eqref{it:dual1} and~\eqref{it:dual4} of
Propositition~\ref{P:dual}, we have that $f_n\in
L^\infty(\CZ_{k,\mu_x},\mu_x)$ for $\mu$-almost every $x$ and every $n\in \N$.
The announced result follows.
\end{proof}
\subsection{Proof of Proposition~\ref{L:ApprNil}}
In order to prove Proposition~\ref{L:ApprNil} we are going to make use of two
ingredients.
The first is a continuity property of dual functions (it follows from Proposition~5.2 and Lemma~5.8 in \cite{HKM}):

\begin{theorem}[{\bf \cite{HKM}}]\label{P:DualContinuity}
Suppose that the topological dynamical system $(Y,S)$ is a topological inverse limit of minimal
$(k-1)$-step  nilsystems, and let $m_Y$ be the unique $S$-invariant measure in $Y$.

Then
  for every $f\in L^{\infty}(m_Y)$, and  $k\in \N$, the function $\cD_k f$ coincides $m_Y$-almost everywhere with a continuous function in $Y$.
\end{theorem}
Let $(X,\X,\mu,T)$ be an ergodic system and let $(Y_{k},\mathcal{Y}_{k}, m_{k},S_{k})$ be a topological inverse limit of minimal nilsystems that is measure theoretically isomorphic to  the factor  $(X,\cZ_{k},\mu,T)$ (see the remark following Theorem~\ref{T:Structure}). With $\pi_k\colon X\to Y_k$ we denote the measure preserving isomorphism that identifies
 $(X,\cZ_{k},\mu,T)$  with $(Y_k,\cY_k, m_k,S_k)$. Using this notation we have:
\begin{corollary}\label{C:DualContinuity}
Let $(X,\X,\mu,T)$ be an ergodic system,  $f\in L^{\infty}(\mu)$, and  $k\in \N$.

Then
   there exists $g\in \CC(Y_{k-1})$ such that
    $\cD_k f$ coincides $\mu$-almost everywhere with the function $g\circ \pi_{k-1}$.
\end{corollary}

\begin{proof}
By part~\eqref{it:dual4} of Proposition~\ref{P:dual} we have that
$\cD_kf=\cD_k\tilde{f}$ where $\tilde{f}=\E(f|\cZ_{k-1})$.
Therefore, we can assume that $f\in L^\infty(\cZ_{k-1},\mu)$. Writing $f=\phi\circ \pi_{k-1}$ for some $\phi\in L^\infty(m_{k-1})$, we have $\cD_k f=(\cD_k \phi)\circ \pi_{k-1}$. The announced result now follows from Theorem~\ref{P:DualContinuity}.
\end{proof}
The second ingredient is Theorem~1.1 of~\cite{HKM}, which gives
a characterization of nilsequences that uses only
local information about the sequence.
To give here the exact statement would necessitate to introduce
definitions and notation that we are not going to use in the sequel, so we choose
to only state an immediate consequence that we need.
\begin{theorem}[{\bf \cite{HKM}}]
 Let  $(a_s(n))_{n\in\N}$   be a collection of sequences indexed
 by a set $S$.

Then for every $k\in \N$ the set of $s\in S$ for which the sequence $(a_s(n))_{n\in\N}$ is a $k$-step
nilsequence belongs to the $\sigma$-algebra
 spanned by  sets of the form
 $A_{l,m,n}=\{s\in S\colon |a_s(m)-a_s(n)|\leq 1/l\}$,  where
$l,m,n\in \N$.
\end{theorem}

Using this,  we  immediately  deduce the following measurability property:

\begin{corollary}\label{cor:HKM2}
Let $(X,\X,\mu,T)$ be a system, $f\in L^\infty(\mu)$, and $k\in\N$.

Then the set
$
A_f=\{x\in X\colon (f(T^nx))_{n\in\N} \text{ is a $k$-step nilsequence} \}
$
is measurable.
\end{corollary}

We are now ready for the proof of Proposition~\ref{L:ApprNil}.
\begin{proof}[Proof of Proposition~\ref{L:ApprNil}]
First notice that if a function $\tilde{f}\in L^\infty(\mu)$ satisfies properties $(i)$ and $(ii)$, then the function
$g=\min(|\tilde{f}|,\norm{f}_{L^\infty(\mu)}) \cdot \text{sign}(\tilde{f})$ has $L^\infty$-norm
bounded by $\norm{f}_{L^\infty(\mu)}$ and still satisfies properties $(i)$ and $(ii)$ (we used here
that $\min(|a_n|,M)\cdot \text{sign}(a_n)$ is a nilsequence if $a_n$ is). So it suffices to find
$\tilde{f}\in L^\infty(\mu)$ that satisfies properties $(i)$ and $(ii)$.

Since, by part~\eqref{it:dual5} of Proposition~\ref{P:dual}, for every $k\in\N$, linear combinations of
functions of the form  $\cD_{k+1,\mu} \phi$ with $\phi\in L^{\infty}(\mu)$  are dense in
$L^1(\cZ_{k,\mu},\mu)$,  we can assume that $f$ is of the  form $\cD_{k+1,\mu} \phi$ for some $\phi\in L^{\infty}(\mu)$. Hence,
 it suffices to show that for every $\phi\in L^\infty(\mu)$ we have   $\mu(A_{\cD_{k+1,\mu} \phi})=1$, where
 $$
 A_{\cD_{k+1,\mu} \phi}=\{x\in X\colon ((\cD_{k+1,\mu} \phi)(T^nx))_{n\in\N} \text{ is a  } k\text{-step nilsequence} \}.
 $$

 Let $\mu=\int \mu_x \ d\mu(x)$ be the ergodic decomposition of the measure $\mu$.
Since by Corollary~\ref{cor:HKM2} the set $A_{\cD_{k+1,\mu} \phi}$ is $\mu$-measurable, it suffices to show that $\mu_x(A_{\cD_{k+1,\mu} \phi})=1$ for $\mu$-almost every $x\in X$.
 By part~\eqref{it:dual1} of Proposition~\ref{P:dual} we have for $\mu$-almost every $x\in X$ that
 $\cD_{k+1,\mu}\phi=\cD_{k+1,\mu_x}\phi$ as functions of $L^\infty(\mu_x)$. As a consequence, it remains to show that $\mu_x(A_{\cD_{k+1,\mu_x} \phi})=1$ for $\mu$-almost every $x\in X$.

We have therefore reduced matters to establishing that $\mu(A_{\cD_{k+1,\mu}}\phi)=1$ for ergodic systems and $\phi\in L^\infty(\mu)$.
Using  Corollary~\ref{C:DualContinuity} and the notation introduced there, we get
 that there exists  a  function $g\in \CC(Y_k)$ such  that for $\mu$-almost every $x\in X$ we have
$$
(\cD_{k+1,\mu} \phi) (x)= g(\pi_k x).
$$
 As a consequence, for $\mu$-almost every $x\in X$, we have
$$
(\cD_{k+1,\mu} \phi)(T^nx)=g(S_k^n\pi_k x) \text{ for every } n\in\N.
$$
Since $(Y_k,S_k)$ is a topological inverse limit of nilsystems and
$g\in \CC(Y_k)$,
for every $y\in Y_k$ the sequence
$(g(S_k^ny))_{n\in\N}$ is a  $k$-step nilsequence.  We conclude that indeed $\mu(A_{\cD_{k+1} \phi})=1$. This completes the proof.
\end{proof}

\section{A characteristic factor for the
highest degree iterate:  Two transformations}\label{S:Char2}

In this section and the next one, we are going to prove  Theorem~\ref{T:CharA}  under the additional
assumption that the function  corresponding to the highest degree polynomial iterate satisfies
the stated orthogonality assumption. For example, if
$\deg(p_1)>\deg(p_i)$ for $i=2,\ldots,\ell$, we assume that
$f_1\bot \CZ_{k,T_1}$ for some $k\in \N$.

In fact our method necessitates that we prove a more general
result (Proposition~\ref{P:CharB}). This result is  also going to be used in Section~\ref{S:Charlower}, when we deal with the
polynomials of lower degree.

However, since the proof is notationally heavy, we present it first in
the case of two commuting transformations. In the next section we
give a sketch of the proof for the general case, focusing on the few
points where the differences are significant.

In this section, we show:

\begin{proposition}\label{P:CharB2}
Let $(X,\X,\mu,T_1,T_2)$ be a system
and $f_1,\ldots, f_m\in L^\infty(\mu)$.  Let   $(\P,\Q)$  be a nice
 ordered  family of pairs of polynomials, with degree $d$
 (all notions are defined in Section~\ref{subsec:families}).

 Then there exists $k=k(d,m)\in \N$ such that: If      $f_1\bot \cZ_{k,T_1}$, then
the averages
\begin{equation}\label{E:fgh}
\frac{1}{N-M}\sum_{n=M}^{N-1}  f_1(T_1^{p_1(n)} T_2^{q_1(n)}x) \cdot \ldots \cdot f_m(T_1^{p_m(n)}
T_2^{q_m(n)}x)
\end{equation}
converge to $0$ in $L^2(\mu)$.
\end{proposition}

Applying this to the nice family $(\P,\Q)$  where $\P=(p_1,0)$ and
$\Q=(0,p_2)$, we get:
\begin{corollary}
Let $(X,\X,\mu,T_1,T_2)$ be a system and $f_1,f_2\in L^\infty(\mu)$.
Let $p_1$ and  $p_2$ be integer polynomials  with
$d=\deg(p_1)>\deg(p_2)$.

Then there exists $k=k(d)$ such that, if
$f_1\perp\CZ_{k,T_1}$, then the averages
$$
\frac{1}{N-M}\sum_{n=M}^{N-1}  f_1(T_1^{p_1(n)}x)\cdot f_2(T_2^{p_2(n)}x)
$$
converge to $0$ in $L^2(\mu)$.
\end{corollary}

\subsection{A simple example}
\label{subsec:simplexample}
We give here a very simple example in order to explain our
strategy. In the appendix  we consider slightly more general
averages and get more precise results (the main drawback of these simpler arguments
is that  they do not allow us to treat any two polynomials with distinct degrees).
 Let $(X,\CX,\mu,T_1,T_2)$ be a system and $f_1,f_2\in
L^\infty(\mu)$.
\begin{claim*}
If  $f_1\perp\CZ_{2,T_1}$, then the averages
\begin{equation}
\label{eq:exampleST}
\frac{1}{N-M}\sum_{n=M}^{N-1}  f_1(T_1^{n^2}x)\cdot
f_2(T_2^{n}x)
\end{equation}
converge to $0$ in $L^2(\mu)$.
\end{claim*}

Using  van der Corput's Lemma it suffices to show that for every
$h_1\in\N$, the averages in $n$ of
$$
 \int f_1(T_1^{n^2}x)\cdot f_2(T_2^nx)\cdot f_1(T_1^{(n+h_1)^2}x)
\cdot f_2(T_2^{n+h_1}x)\ d\mu(x)
$$
converge to $0$.
After composing with $T_2^{-n}$ and using the  Cauchy-Schwarz inequality,
we reduce matters to showing that the averages in $n$ of
$$
f_1(T_1^{n^2}T_2^{-n}x)\cdot f_1(T_1^{(n+h_1)^2}T_2^{-n}x)
$$
converge to $0$ in $L^2(\mu)$.
Using van der Corput's Lemma  one more time, we reduce matters to showing that for
every fixed $h_1\in \N$, for  every large enough $h_2\in\N$,   the averages in $n$ of
\begin{multline*}
 \int f_1(T_1^{n^2}T_2^{-n}x)\cdot f_1(T_1^{(n+h_1)^2}T_2^{-n}x)\cdot
f_1(T_1^{(n+h_2)^2}T_2^{-n-h_2}x)\cdot
f_1(T_1^{(n+h_1+h_2)^2}T_2^{-n-h_2}x)\,
d\mu(x)
\end{multline*}
converge to $0$, or equivalently, that the averages in $n$ of
\begin{multline}
\label{eq:finalexample}
 \int f_1(x)\cdot f_1(T_1^{2nh_1+h_1^2}x)\cdot
f_1(T_1^{2nh_2+h_2^2}T_2^{-h_2}x)\cdot
f_1(T_1^{2n(h_1+h_2)+(h_1+h_2)^2}T_2^{-h_2}x)
\, d\mu(x)
\end{multline}
converge to $0$.

The important property of this last average is that it involves only constant iterates of the
transformation $T_2$ (for
$h_1,h_2$ fixed). Therefore, we can apply the known results about the
convergence of averages of a single transformation.
 It  follows from Theorem~\ref{T:linear}  that the averages in $n$
of~\eqref{eq:finalexample}
converge to $0$ for all  $h_1,h_2\in\N$ such that
 the linear polynomials $2h_1n, 2h_2n, 2(h_1+h_2)n$ are distinct,
that is, for all $h_1,h_2\in\N$ with $h_1\neq h_2$. The claim follows.

We will come back to this example in Section~\ref{subsec:simplexample2}.

\subsection{Families of pairs and their type}
In this subsection we follow \cite{BL}  with some changes on the notation, in order to
define the type of a family of pairs of polynomials.
\label{subsec:families}
\subsubsection{Families of pairs of polynomials}
 Let $m\in\N$. Given two ordered families of polynomials
 $$
 \P=(p_1,\ldots,p_m), \quad \Q=(q_1,\ldots,q_m)
 $$
 we define the \emph{ordered family of pairs of polynomials} $(\P,\Q)$ as follows
$$
(\P,\Q)=\big((p_1,q_1),\ldots,(p_m,q_m)\big).
$$
The reader is advised to think of this family as an efficient way to
record  the polynomial iterates that appear in \eqref{E:fgh}.

The maximum of the degrees of the polynomials in the families $\P$
and $\Q$ is called
the \emph{degree of the family} $(\P,\Q)$.

For convenience of exposition, if pairs of constant polynomials appear in $(\P,\Q)$ we remove them, and
henceforth we assume:
\begin{itemize}
\item[] \em
All families $(\P,\Q)$ that we consider do not contain pairs of
constant polynomials.
\end{itemize}

\subsubsection{Definition of type}\label{SS:type}

 We fix an integer $d\geq 1$ and
restrict ourselves to families $(\P,\Q)$ of degree $\leq d$.

We say that two polynomials $p,q\in \Z[t]$ are \emph{equivalent}, and write $p\sim q$,
if they have the same degree and the same leading coefficient. Equivalently,  $p\sim q$ if and only if $\deg(p-q)< \min\{\deg(p),\deg(q)\}$

We define $\Q'$ to be the following set (possibly empty)
$$
\Q'=\{q_i\in \Q\colon p_i \text{ is constant}\}.
$$

For $i=1,\ldots, d$, let $w_{1,i}$, $w_{2,i}$ be the number of distinct non-equivalent classes of polynomials of degree $i$ in $\P$ and  $\Q'$ correspondingly.

We define the \emph{(matrix) type} of the family $(\P,\Q)$ to be the $2\times d$  matrix
$$
\begin{pmatrix}
w_{1,d}& \ldots &  w_{1,1}\\ w_{2,d}& \ldots& w_{2,1}
\end{pmatrix}.
$$
If $\Q'$ is empty, then all the elements of the second row are taken to be $0$.
For example, with $d=4$, the family
$$
\big((n^2, n^4),\ (n^2+n,n), \ (2n^2, 2n), \ (0,n^3), \ (0,n)\big)
$$
has  type
$$
\begin{pmatrix}
0& 0& 2& 0\\ 0& 1& 0 & 1
\end{pmatrix}.
$$
We order the types lexicographically; we  start by comparing the first element of the first row of each matrix, and after going
through all the elements of the first row,  we compare the  elements of the second row of each matrix,  and so on.
In symbols: given two $2\times d$ matrices $W=(w_{i,j})$
and $W'=(w'_{i,j})$, we say that $W> W'$ if:  $w_{1,d}>w'_{1,d}$, or $w_{1,d}=w'_{1,d}$ and $w_{1,d-1}>w'_{1,d-1}$, $\ldots$, or $w_{1,i}=w'_{1,i}$ for $i=1,\ldots,d$ and $w_{2,d}>w'_{2,d}$, and so on.

For example
\begin{multline*}
\begin{pmatrix}
 2& 2\\ 0& 0
\end{pmatrix}
>
\begin{pmatrix}
 2& 1\\ \star &\star
\end{pmatrix} >
\begin{pmatrix}
 2& 0\\ \star & \star
\end{pmatrix} >
\begin{pmatrix}
 1& \star\\ \star& \star
\end{pmatrix} >
\begin{pmatrix}
 0& \star\\ \star& \star
\end{pmatrix}
\geq
\begin{pmatrix}
 0& 0\\ \star& \star
 \end{pmatrix}\geq
 \begin{pmatrix}
 0& 0\\ 0& \star
\end{pmatrix}
\geq
 \begin{pmatrix}
 0& 0\\ 0& 0
\end{pmatrix}.
\end{multline*}
where in the place of the stars one can put any collection of non-negative integers.

An important observation is that although for a given type $W$ there is an
infinite number of possible types $W'<W$, we have
\begin{lemma}
\label{lem:decreasing}
Every decreasing sequence of
 types of families of polynomial pairs is stationary.
\end{lemma}
Therefore, if some operation reduces the type,
then after a finite number of repetitions it is going to terminate. This is the basic principle
behind all  the PET induction arguments used in the literature and in this article.

\subsection{Nice families and the van der Corput operation.}
In this subsection we  define a class of families of pairs of polynomials
that we are going to work with in the sequel, and an important operation that preserves
such families and reduces their type.
\subsubsection{Nice families} Let $\P=(p_1,\ldots,p_m)$ and $\Q=(q_1,\ldots,q_m)$.
\begin{definition*}
We call the ordered  family of pairs of polynomials  $(\P,\Q)$ \emph{nice} if
\begin{enumerate}
\item\label{it:nice1}
 $\deg(p_1)\geq \deg(p_i)$ for $i=1,\ldots,m$\ ;
\medskip

\item\label{it:nice2}
  $\deg(p_1)>\deg(q_i)$ for $i=1,\ldots,m$\ ;
\medskip

\item \label{it:nice3}
 $\deg(p_1-p_i)>\deg (q_1-q_i)$ for $i=2,\ldots,m$.
\end{enumerate}
(Notice that a consequence of~\eqref{it:nice3} is that $p_1-p_i\neq \text{const}$ for $i=2,\ldots,m$.)
\end{definition*}
As an example, if a nice family consists of $m$ pairs of polynomials and has degree $1$, then we have:  $\deg(p_1)=1$,  $\deg(p_i)\leq 1$,
  $\deg(q_i)=0$ for  $i=1,\ldots,m$,  and  $\deg(p_1-p_i)=1$ for $i=2,\ldots,m$.  It follows that the type of this family is
\begin{equation}
\label{eq:degre1}
 \begin{pmatrix}
 0& \cdots & 0& k \\ 0& \cdots & 0 & 0
\end{pmatrix}
\end{equation}
for some $ k\in \N$ with $k\leq  m$.

\subsubsection{The van der Corput operation}
\label{subsec:vdcoperation}
Given a family $\P=\big(p_1,\ldots,p_m\big)$,  $p\in \Z[t]$, and $h\in\N$, we define
 $$
 S_h\P=(p_1(n+h),\ldots,p_m(n+h)) \text{ and }
 \P-p=\big(p_1-p,\ldots,p_m-p\big).
 $$
Given a family of   pairs of polynomials $(\P,\Q)$, a pair  $(p,q)\in (\P,\Q)$, and $h\in\N$,
 we define the following operation
 $$
(p,q,h)\vdc(\P,\Q) =(\tilde{P},\tilde{Q})^\ast
 $$
 where
 $$
  \tilde{P}=(S_h\P-p,\P-p), \quad \tilde{Q}=(S_h\Q-q,\Q-q),
 $$
and  $^\ast$ is the operation that removes all pairs of constant polynomials from a given family of
pairs of polynomials.
A more explicit form of the family $ (p,q,h)\vdc(\P,\Q)$ is
 \begin{multline*}
 \big( (S_hp_1-p,S_hq_1-q),\ldots,(S_hp_m-p,S_hq_m-q),
(p_1-p,q_1-q),\ldots,(p_m-p,q_m-q)\big)^\ast.
\end{multline*}
 Notice that if the family $(\P,\Q)$ has degree $d$ and contains $m$ pairs of polynomials, then for every $h\in\N$,
 the family $(p,q,h)\vdc(\P,\Q)$ has degree at most $d$ and contains at most $2m$ pairs of polynomials.

\subsection{An example}
\label{SSS:example2}
In order to explain our method we  give an example that is somewhat more complicated
than the example of Section~\ref{subsec:simplexample}.
 When we study the limiting behavior of the averages
 $$
\frac{1}{N-M}\sum_{n=M}^{N-1} f_1(T_1^{n^3}x)\cdot f_2(T_2^{n^2}x),
 $$
   we define $\P=(n^3,0)$, $\Q=(0,n^2)$, and  introduce the family of pairs of polynomials
 $$
 (\P,\Q)=\big((n^3,0),(0,n^2)\big).
 $$
 This family is nice and has  type
 $\left(
\begin{smallmatrix}
1& 0& 0\\ 0& 1& 0
\end{smallmatrix}\right)$.
  Applying the vdC operation with    $(p,q)=(0,n^2)$  and $h\in\N$,
we arrive to the new family
$$
(0,n^2,h)\vdc(\P,\Q) =(\tilde{P}_h,\tilde{Q}_h)
 $$
where
$$
  \tilde{P}_h=((n+h)^3,0,n^3,0), \quad \tilde{Q}_h=(-n^2,2hn+h^2,-n^2,0)\ ;
 $$
then the corresponding family of pairs  is
 $$
 \big( \big((n+h)^3,-n^2\big),\big(0,2hn+h^2\big), \big(n^3,-n^2\big)\big).
 $$
The important point is that for every $h\in\N$  this new family  is also nice and has smaller type,
namely $\left(\begin{smallmatrix} 1& 0& 0\\ 0& 0& 1
\end{smallmatrix}\right)$.
Translating back to ergodic theory, we get the averages
 $$
 \frac{1}{N-M}\sum_{n=M}^{N-1}
 \tilde{f}(T_1^{(n+h)^3}T_2^{-n^2}x)\cdot
\tilde{g}(T_2^{2hn+h^2}x)
 \cdot \tilde{h}(T_1^{n^3}T_2^{-n^2}x)
 $$
 for some choice of functions $\tilde{f},\tilde{g},\tilde{h}\in
L^\infty(\mu)$. Concerning the choice of these functions, the only important thing for our purposes is that    $\tilde{f}=f_1$.

\subsection{The general strategy}
As was the case in the previous example, we are going to show that if we are given a nice
 family $(\P,\Q)$ with $\deg (p_1)\geq 2$, then  it is always  possible  to find
appropriate $(p,q)\in (\P,\Q)$ so that for all large enough $h\in \N$ the
 operation $(p,q,h)\vdc$ leads to a  nice family that has smaller type.
Our  objective is, after successively applying the   operation
$(p,q,h)\vdc$,
to finally get
  nice families  of degree $1$, and thus with matrix type of the form~\eqref{eq:degre1}.

 Translating this back to  ergodic theory,  we get multiple ergodic averages (with certain parameters)
where: $(i)$ only linear iterates of the transformation  $T_1$ appear
and the iterates of $T_2$  are constant, and $(ii)$ the
``first'' iterate of $T_1$ is applied to the ``first'' function of the original average.
The advantage now is that the limiting behavior of such averages can be treated easily using
the well developed theory   of multiple ergodic averages involving a single transformation.

Let us remark though that in practice
this process becomes  cumbersome very quickly.
For instance, in the example of Section~\ref{SSS:example2}, for every $h\in \N$, the next
$(p_h,q_h,h')\vdc$ operation
uses $(p_h,q_h)=(0,2hn+h^2)$ and leads to a family with matrix type
$\left(\begin{smallmatrix} 1& 0& 0\\ 0& 0& 0   \end{smallmatrix}\right)$. The subsequent  vdC operation
leads to a family with matrix type   $\left(\begin{smallmatrix} 0& 7& 0\\ 0& 0& 0   \end{smallmatrix}\right).$
 One then has to apply the vdC operation a huge number of  times (it is not even easy to  estimate this number)
in order  to reduce the matrix type to the form~\eqref{eq:degre1}.
 So even in the case of two commuting transformations,
it is practically impossible to spell out the details of how this process works when both
polynomial iterates are non-linear.

\subsection{Choosing a good vdC operation}
The next lemma is the key ingredient used to  carry out the previous  plan.
 To prove it  we are going to use freely  the following easy to prove fact:
If $p,q$ are two non-constant polynomials   and $p\sim q$, then
$\deg (p-q)\leq \deg (p)-1$,
and with the possible exception of one $h\in\Z$ we have
$\deg (S_hp-q)= \deg (p)-1$.

\begin{lemma}\label{L:reduceA}
Let $(\P,\Q)$ be  a nice family of pairs of polynomials,  and suppose that $\deg(p_1)\geq 2$.

 Then there exists
 $(p,q)\in (\P,\Q)$, such that for every large enough $h\in\N$,  the
 family $(p,q,h)\vdc(\P,\Q)$   is nice and has strictly smaller type than that of $(\P,\Q)$.
\end{lemma}
\begin{proof}
Let  $\P=(p_1,\ldots,p_m)$, $\Q=(q_1,\ldots,q_m)$, then for $(p,q)\in (\P,\Q)$ and  $h\in\N$
the  family $(p,q,h)\vdc(\P,\Q)$  is an ordered family of
pairs of polynomials, all of them of the form
$$
(S_hp_i-p,S_hq_i-q), \  \text{ or } \ (p_i-p, q_i-q).
$$

\subsubsection*{We choose $(p,q)$ as follows:}
  If $\Q'$ is non-empty, then we take $p=0$ and let $q$ to be a polynomial
  of smallest degree in $\Q'$. Then the first row of the matrix type
  remains unchanged, and the second row will get ``reduced'', leading
  to a smaller matrix type. Suppose now that $\Q'$ is empty. If $\P$
  consists of a single polynomial $p_1$, then we choose
  $(p,q)=(p_1,q_1)$ and the result follows. Therefore, we can assume
  that $\P$ contains a polynomial other than $p_1$. We consider two
  cases. If $p\sim p_1$ for all $p\in \P$, then we choose
  $(p,q)=(p_1,q_1)$.
Otherwise, we choose
  $(p,q)\in (\P,\Q)$ such that $p\nsim p_1$  and $p$  is a polynomial in $\P$ with minimal degree (such a choice exists since $p_1$ has the highest degree in $\P$).

In all cases, for every $h\in \N$,  the first row of the matrix type
of $(p,q,h)\vdc(\P,\Q)$ is ``smaller'' than that of $(\P,\Q)$,
and as a consequence the new family has strictly smaller  type.

It remains to verify that for every  large enough   $h\in\N$   the ordered family of pairs of polynomials
$(p,q,h)\vdc(\P,\Q)$ is nice.
We remark that,  by construction,  the first polynomial pair in this family is
$(S_hp_1-p, S_hq_1-q)$.

\begin{claim*}
Property~\eqref{it:nice1} holds for every $h\in \N$.
\end{claim*}
 Equivalently, we claim that
 $$
 \deg (S_hp_1-p)\geq \max\{\deg (p_i-p),\deg (S_hp_i-p)\}
  \text{ for } i=1,\ldots,m.
    $$
  If $p\nsim p_1$, then  $\deg (S_hp_1-p)=\deg (p_1)$ and the claim follows from our assumption  $\deg (p_1)\geq \deg (p_i)$  for $i=1,\ldots,m$. If $p\sim p_1$, then by the choice of the polynomial $p$ we have  $p=p_1$ and  $p\sim p_i$ for $i=1,\ldots,m$. As a result,
  $\deg (S_hp_1-p)=\deg (p_1)-1$ and  $\max\{\deg (p_i-p), \deg (S_hp_i-p)\}\leq \deg (p_1)-1$,
  proving the claim.

\begin{claim*} Property~\eqref{it:nice2} holds for every $h\in \N$. \end{claim*}
Equivalently, we claim that
$$
 \deg (S_hp_1-p)> \max\{\deg (q_i-q),\deg (S_hq_i-q)\}
  \text{ for } i=1,\ldots,m.
    $$
       If $p\nsim p_1$, then
     $\deg (S_hp_1-p)=\deg (p_1)$ and the claim follows since  by assumption $\deg (p_1)> \deg (q_i)$  for $i=1,\ldots,m$.
    If $p\sim p_1$, then by the choice of $p$ we have   $(p,q)=(p_1,q_1)$  and $p\sim p_i$ for $i=1,\ldots,m$.
     By hypothesis we have
\begin{equation}\label{E:q2w}
\deg (q_i-q_1)< \deg (p_i-p_1)\leq \deg (p_1)-1=\deg (S_hp_1-p_1).
\end{equation}
It remains to  verify that $\deg (S_hp_1-p_1)>\deg (S_hq_i-q_1)$. To see this we   express $S_hq_i-q_1$ as $(S_hq_i-q_i)+(q_i-q_1)$. If $q_i$ is non-constant, then the first polynomial has degree $\deg (q_i)-1<\deg (p_1)-1=\deg (S_hp_1-p_1)$. If $q_i$ is constant, then it has degree $0<\deg (p_1)-1=\deg (S_hp_1-p_1)$ (we used here that $\deg (p_1)\geq 2$).
Furthermore, by \eqref{E:q2w} the second polynomial has
degree $\deg (q_i-q_1)< \deg (S_hp_1-p_1)$.
This proves the claim.

\begin{claim*}
Property~\eqref{it:nice3} holds for all except finitely many values
of $h$.
\end{claim*}
Equivalently, we claim that
 $$
 \deg (S_hp_1-S_hp_i)> \deg (S_hq_1-S_hq_i),
  \text{ for } i=2,\ldots,m,
    $$
    and
 $$
 \deg (S_hp_1-p_i)> \deg (S_hq_1-q_i),
  \text{ for } i=1,\ldots,m.
    $$
    The first estimate follows immediately from our hypothesis $\deg (p_1-p_i)> \deg (q_1-q_i)$
   for  $i=2,\ldots,m$. It remains to verify the second estimate.   If $p_i\nsim p_1$, then
     $\deg (S_hp_1-p_i)=\deg (p_1)$ and the claim follows since  by hypothesis $\deg (p_1)> \deg (q_i)$  for $i=1,\ldots,m$. Suppose now that  $p_i\sim p_1$. Then $\deg (S_hp_1-p_i)=\deg (p_1)-1$, with the
     possible exception of one $h\in\N$ (hence we get at most $m-1$ exceptional values of $h$). So it remains to verify that $\deg (S_hq_1-q_i)< \deg (p_1)-1$.
To see this we express  $S_hq_1-q_i$ as $(S_hq_1-q_1)+(q_1-q_i)$. The first polynomial
  has degree $\deg (q_1)-1<\deg (p_1)-1$ if $q_1$ is non-constant, and degree $0<\deg (p_1)-1$
 (we used that $\deg (p_1)\geq 2$) if $q_1$ is constant.
The second polynomial  has  degree $\deg (q_1-q_i)<\deg (p_1-p_i)\leq \deg (p_1)-1$ since $p_i\sim p_1$.
This establishes the claim and completes the proof.
\end{proof}

We say that  a  subset of $\N^k$  is \emph{good} if it is of the form
\begin{equation}\label{E:parameters}
\{h_1\geq c_1,\  h_2\geq c_{2}(h_1),\ldots, \ h_k\geq c_{k}(h_1,\ldots,h_{k-1})\}
\end{equation}
for some $c_i\colon \N^{i-1}\to \N$.
The next lemma will be used in order to prove that the level $k$ of
the characteristic factors $\cZ_{k,T_i}$ considered in
Theorem~\ref{T:CharA} depends only on the number and the maximum
degree of the polynomials involved.

\begin{lemma}\label{L:k(d,m)}
Let $(\P,\Q)$ be a nice family with degree $d\geq 2$ that  contains $m$ pairs of polynomials.
Suppose that we successively apply the  $(p,q,h)\vdc$ operation for appropriate choices
of $p,q\in \Z[t]$ and $h\in\N$, as described in the previous lemma,
each time getting a  nice family of pairs of polynomials with strictly smaller matrix  type.

Then after a finite  number of  operations we get, for a good set of parameters,
  nice families of pairs of polynomials of degree $1$.
Moreover, the  number of operations needed
can be bounded by
a function of $d$ and $m$ alone.
\end{lemma}
\begin{remark}
The exact dependency on $d$ and $m$ seems neither easy nor very useful to pin down; it appears to be a tower of exponentials the length of which depends on $d$ and  $m$.
\end{remark}
\begin{proof}
We fix $d\geq 2$.
The first statement follows immediately from Lemma~\ref{lem:decreasing}.

We denote by $W(\P,\Q)$ the matrix type of a given family $(\P,\Q)$, and
by $N(\P,\Q)$ the number of operations mentioned in the statement needed to get the particular matrix type.

First we claim that  it suffices to show the following:
For every nice family $(\P,\Q)$ with degree $d$, containing at most $m$ polynomials, we have
\begin{equation}\label{E:TypeEstim}
N(\P,\Q)\leq f(W(\P,\Q),m)
\end{equation}
for some function $f$, with the obvious domain, and range in the non-negative integers.
Indeed, since there exists a finite number of possible matrix types for a  family $(\P,\Q)$
with degree at most $d$ and containing  at most $m$ polynomials (in fact there are at most $(m+1)^{2d}$
such matrix types), we have
$$
N(\P,\Q)\leq F(d,m)=\max_{W}(f(W,m))
$$
where $W$ ranges over all possible matrix types of nice families with degree $d$  that contain at most $m$
pairs of polynomials. This proves our claim.

Next we turn our attention to establishing \eqref{E:TypeEstim}.
Let
$$
W_1=\begin{pmatrix} 0& \cdots &0& 1\\ 0& \cdots &0& 0    \end{pmatrix}.
$$
Notice that $W_1$ is the smallest
matrix type (with respect to the order introduced before) that can appear as a first coordinate entry in the domain of $f$.
We define $f$  recursively as follows:
\begin{equation}\label{E:deff}
f(W_1,m)=0 \text{ for every } m\in \N, \quad \text{and}\quad
f(W,m)=\max_{W'<W}( f(W',2m))+1
\end{equation}
where the maximum is taken over the finitely many possible
 matrix types $W'$ of nice families  of degree at most $d$ that contain at
most $2m$ pairs of polynomials.

Since every $(p,q,h)\text{-vdC}$ operation of the previous lemma preserves nice families of pairs of polynomials,
does not increase their degree, reduces their matrix type, and
  at most doubles the number $m$ of (non-constant) pairs of polynomials in the family,
a straightforward induction on the type $W(\P,\Q)$ establishes \eqref{E:TypeEstim} with $f$
defined by \eqref{E:deff}.
This completes the proof.
\end{proof}

\subsection{Proof of Proposition~\ref{P:CharB2}}\label{SS:CharB2}

Let  $(\P,\Q)$ be a nice family of pairs of polynomials where $\P=(p_1,\ldots,p_m)$ and
$\Q=(q_1,\ldots,q_m)$ and let $d$ be the degree of this family.
We remind the reader that our goal is to show that
there exists $k=k(d,m)\in \N$ such that: If      $f_1\bot \cZ_{k,T_1}$, then
the averages of
\begin{equation}\label{E:fgh'}
 f_1(T_1^{p_1(n)} T_2^{q_1(n)}x) \cdot \ldots \cdot f_m(T_1^{p_m(n)}
T_2^{q_m(n)}x)
\end{equation}
converge to $0$  in $L^2(\mu)$.

\medskip\noindent\textbf{(a)}
Suppose first that  $\deg (p_1)=1$. Since the family $(\P,\Q)$ is nice, we have
$\deg (p_i)=1$ for $i=1,\ldots,m$, all the polynomials $q_1,\ldots,q_m$ are constant, and $p_1-p_i\neq\text{const}$
for $i=1,\ldots,m$. In other words we are reduced to studying the limiting behavior of the
  averages in $n$ of
$$
 f_1(T_1^{a_1n+b_1}T_2^{c_1}x)\cdot f_2(T_1^{a_2n+b_2}T_2^{c_2}x)\cdot\ldots\cdot
f_m(T_1^{a_m n+b_m}T_2^{c_m}x)$$
where $a_i,b_i,c_i\in\Z$, $a_i\neq 0$, for $i=1,\ldots,m$,  and $a_1\neq a_i$ for $i=2,\ldots,m$. Suppose that
$f_1\bot \cZ_{m-1,T_1}$, then  also $T_2f_1\bot \cZ_{m-1,T_1}$ (since
$T_1$ and $T_2$ commute).
By Theorem~\ref{T:linear}  the previous averages
converge to $0$ in $L^2(\mu)$, and as a consequence the same holds for the averages of \eqref{E:fgh'}.

\medskip\noindent\textbf{(b)}
Suppose now that $\deg (p_1)\geq 2$. Our objective is to
repeatedly use  van der Corput's Lemma in order
to reduce matters to the previously established linear case.

To begin with, using van der Corput's Lemma we see that in order to
establish  convergence to $0$ for the averages of~\eqref{E:fgh'}, it suffices to show that,
for every sufficiently large $h\in \N$,
the averages in $n$ of
\begin{multline*}
\int f_1(T_1^{p_1(n+h)} T_2^{q_1(n+h)}x) \cdot \ldots \cdot
f_m(T_1^{p_m(n+h)} T_2^{q_m(n+h)}x)\cdot \\
{f}_1(T_1^{p_1(n)} T_2^{q_1(n)}x) \cdot \ldots \cdot
{f}_m(T_1^{p_m(n)} T_2^{q_m(n)}x)\ d\mu
\end{multline*}
converge to $0$.
We compose with $T_1^{-p(n)}T_2^{-q(n)}$, where $(p,q)\in (\P,\Q)$ is chosen as in Lemma~\ref{L:reduceA}, and use the Cauchy-Schwarz inequality.
  This  reduces matters to showing that, for every sufficiently large
$h\in \N$,
the averages in $n$  of
 \begin{multline}\label{E:pop}
 f_1(T_1^{p_1(n+h)-p(n)} T_2^{q_1(n+h)-q(n)}x) \cdot \ldots \cdot
f_m(T_1^{p_m(n+h)-p(n)} T_2^{q_m(n+h)-q(n)}x)
\cdot \\
{f}_1(T_1^{p_1(n)-p(n)} T_2^{q_1(n)-q(n)}x) \cdot \ldots \cdot
{f}_m(T_1^{p_m(n)-p(n)} T_2^{q_m(n)-q(n)}x)
\end{multline}
 converge to $0$ in $L^2(\mu)$.
We remove the functions that happen to be composed with constant
iterates of  $T$ and $S$, since they do not affect  convergence to
$0$. This corresponds to the operation $^\ast$  defined in
Section~\ref{subsec:vdcoperation}.
We get  multiple ergodic averages that correspond  to the families of polynomials
 $(p,q,h)\vdc(\P,\Q)$; our goal is to show convergence to $0$ in $L^2(\mu)$ for every large enough  $h\in\N$ .

 By Lemma~\ref{L:reduceA}, for every large enough  $h\in\N$,  the
family $(p,q,h)\vdc(\P,\Q)$ is nice, and its first pair  is
 $(p_1(n+h)-p(n),q_1(n+h)-q(n))$.
Notice also that, in \eqref{E:pop}  the iterate  $T_1^{p_1(n+h)-p(n)}T_2^{q_1(n+h)-q(n)}$ is
applied to the function $f_1$.
We consider two cases depending on the degree of the polynomial $p_1(n+h)-p(n)$.

\medskip\noindent\textbf{(b$_1$)}  If $\deg (p_1(n+h)-p(n)) =1$, then
  we are reduced to the case (a) studied before. As we explained, if
$f_1\bot \cZ_{2m,T_1}$,
then the averages \eqref{E:pop} converge to $0$ in $L^2(\mu)$  for every large enough $h\in \N$.
As a consequence,  the averages \eqref{E:fgh'} converge to $0$ in $L^2(\mu)$.

 \medskip\noindent\textbf{(b$_2)$} If $\deg (p_1(n+h)-p(n))
\geq 2$, then  we can iterate the  ``van der Corput  operation''.
By Lemma~\ref{L:k(d,m)}, there  exists $k=k(d,m)\in \N$, such that after at most $k$ such operations,
   we arrive to  averages  involving, for a good set of parameters $G$ of the form \eqref{E:parameters},  nice families of pairs of polynomials of the type studied in part (a). More precisely, we are left with
  studying the averages in $n$ of
  \begin{equation}\label{E:fgh''}
 g_1(T_1^{a_1n+b_1}T_2^{c_1}x)\cdot\ldots\cdot
g_{\tilde{m}}(T_1^{a_{\tilde{m}} n+b_{\tilde{m}}}T_2^{c_{\tilde{m}}}x)
\end{equation}
where the functions $g_i$, and the integers $a_i, b_i, c_i$, depend on $k$ parameters,
and satisfy: (i) $g_1=f_1$ (this last condition follows easily by the definition of the vdC-operation),
and (ii) $a_1\neq a_i$ for $i\in \{2,\ldots,\tilde{m}\}$.
  Our goal is to show convergence to $0$ in $L^2(\mu)$ for the averages of \eqref{E:fgh''} for   this good set of  parameters $G$.
Then repeated uses of van der Corput's Lemma show that the averages of \eqref{E:fgh'} converge to $0$ in $L^2(\mu)$.

We proceed to establish our goal.
Since  the number of functions involved  at most doubles after each vdC-operation
  is performed, we have
$\tilde{m}\leq 2^{k}m$.
It follows by Theorem~\ref{T:linear} and properties (i) and (ii) above, that if $f_1\bot \cZ_{2^km,T_1}$,
 then for every choice of parameters in the  ``good'' set $G$,
 the averages of \eqref{E:fgh''} converge to $0$ in $L^2(\mu)$, establishing our goal.
 As a consequence, the averages of \eqref{E:fgh'} converge to $0$ in $L^2(\mu)$.

   Concluding,  if $f_1\bot\cZ_{2^km,T_1}$, then in all
cases we showed that the averages of~\eqref{E:fgh'} converge to $0$ in $L^2(\mu)$.
 This completes the proof of Proposition~\ref{P:CharB2}.\qed

\section{A characteristic factor for the
highest degree iterate:  The general case}\label{S:Char}

The next proposition is the generalization of
Proposition~\ref{P:CharB2} to the case of an arbitrary number of
transformations.
Its proof is very similar to the proof of Proposition~\ref{P:CharB2}
that was given in the previous section. To avoid unnecessary repetition, we
 define the concepts needed in the proof of  Proposition~\ref{P:CharB},
and then only summarize its proof providing details only when non-trivial
modifications of the arguments used in the previous section are needed.

\begin{proposition}\label{P:CharB}
 Let $(X,\X,\mu,T_1,\cdots,T_\ell)$ be a system,
and $f_1,\ldots,f_m\in L^\infty(\mu)$.
Suppose that  $(\P_1,\ldots,\P_\ell)$  is a nice  ordered  family of  $\ell$-tuples of polynomials
with degree $d$ (all notions are  defined below).

Then there exists $k=k(d,\ell,m)\in \N$   such that:
If     $f_1\bot \cZ_{k,T_1}$, then the averages
$$
 \frac{1}{N-M}\sum_{n=M}^{N-1}  f_1(T_1^{p_{1,1}(n)}\cdots T_\ell^{p_{\ell,1}(n)}x) \cdot \ldots \cdot f_m(T_1^{p_{1,m}(n)}\cdots T_\ell^{p_{\ell,m}(n)}x)
$$
 converge to $0$ in $L^2(\mu)$.
\end{proposition}
Applying this result to the family $(\P_1,\ldots,\P_\ell)$ where
$\P_1=(p_1,0,\ldots,0)$, $\P_2=(0,p_2,\ldots,0)$, ...
$\P_\ell=(0,\ldots,0,p_\ell)$, we get:
\begin{corollary}
\label{cor:charZT1}
 Let $(X,\X,\mu,T_1,\cdots,T_\ell)$ be a system,
and $f_1,\ldots,f_\ell\in L^\infty(\mu)$. Let $p_1,\cdots,p_\ell$ be
integers polynomials  with distinct degrees and highest degree
$d=\deg(p_1)$.

Then there exists $k=k(d,\ell)$ such that: If $f_1\perp \CZ_{k,T_1}$,
then the averages
$$
\frac{1}{N-M}\sum_{n=M}^{N-1}  f_1(T_1^{p_1(n)}x)\cdot\ldots\cdot f_\ell(T_\ell^{p_\ell(n)}x)
$$
converge to $0$ in $L^2(\mu)$.
\end{corollary}

\subsection{Families of $\ell$-tuples and their types}
In this subsection we follow \cite{BL} with some changes in the notation.
\subsubsection{Families of $\ell$-tuples of polynomials}
Let $\ell,m\in\N$. Given $\ell$ ordered families of polynomials
$$
\P_1=(p_{1,1},\ldots,p_{1,m}) ,\ldots, \P_\ell=(p_{\ell,1},\ldots,p_{\ell,m})
$$
we define an \emph{ordered family of $m$ polynomial $\ell$-tuples} as follows
$$
(\P_1,\ldots,\P_\ell)=\big((p_{1,1},\ldots,p_{\ell,1}),\ldots,(p_{1,m},\ldots,p_{\ell,m})\big).
$$
The reader is advised to think of this family as an efficient way to
record the polynomial iterates that appear in the average of
$$
 f_1(T_1^{p_{1,1}(n)}\cdots T_\ell^{p_{\ell,1}(n)}x) \cdot \ldots
\cdot f_m(T_1^{p_{1,m}(n)}\cdots T_\ell^{p_{\ell,m}(n)}x).
$$
The maximum of the degrees of the polynomials in the families
$\P_1,\ldots,\P_\ell$ is called \emph{the degree of  the family} $(\P_1,\ldots,\P_\ell)$.

For convenience of exposition, if  $\ell$-tuples of constant  polynomials  appear in $(\P_1,\ldots,\P_\ell)$ we remove them, and
henceforth we assume:

\emph{All families  $(\P_1,\ldots, \P_\ell)$  that we consider do not contain $\ell$-tuples of constant
polynomials.}

\subsubsection{Definition of type}
We fix $d\geq 1$ and restrict ourselves to families of degree $\leq d$.

For $i=1,\ldots, \ell$, we define $\P_i'$ to be the following set (possibly empty)
$$
\P_i'=\{ \text{non-constant } p_{i,j} \in \P_i\colon p_{i',j} \text{ is constant for } i'<i\}.
$$
(It follows that $\P_1'$ is the set of non-constant polynomials is $\P_1$.)

For $i=1,\ldots, \ell$ and $j=1,\ldots, d$, we  let $w_{i,j}$ be the number of distinct non-equivalent classes of polynomials of degree $j$ in the family  $\P_i'$.

We define the \emph{(matrix) type} of the family $(\P_1,\ldots, \P_\ell)$ to be the matrix
$$
\begin{pmatrix}
w_{1,d}& \ldots &  w_{1,1}\\ w_{2,d}& \ldots& w_{2,1}\\ \vdots & \ldots &\vdots \\
w_{\ell,d}& \ldots& w_{\ell,1}
\end{pmatrix}.
$$

For example, let $d=4$,
 and consider the family of triples of polynomials
\begin{multline*}
\big((n^2, n^4,n^4), \ (n^2+n,3n^3,0), \ (2n^2, 0,2n), \ (n,2n,0), \\
\ (0,n^3,n^4), \  (0,2n^3, n^2),
\ (0,0,n^3), \ (0,0,n^3+1)\big).
\end{multline*}
 Since
$$
 \P_1'=\{n^2,n^2+n,2n^2,n\},\quad   \P_2'=\{n^3,2n^3\}, \quad
 \P_3'=\{ n^3,n^3+1 \},
$$
the   type of this family is
$$
\begin{pmatrix}
0& 0& 2& 1\\ 0& 2& 0 & 0\\0& 1& 0&0
\end{pmatrix}.
$$

As in Section~\ref{SS:type}, we order these  types  lexicographically: Given two $\ell\times d$ matrices $W=(w_{i,j})$
and $W'=(w'_{i,j})$, we say that the first is bigger than the second,
and write $W>W'$, if $w_{1,d}>w'_{1,d}$, or $w_{1,d}=w'_{1,d}$ and $w_{1,d-1}>w'_{1,d-1}$, $\ldots$,
or $w_{1,i}=w'_{1,i}$ for $i=1,\ldots,d$ and $w_{2,d}>w'_{2,d}$, and so on.
As for the types of families of pairs, we have:
\begin{lemma}
\label{lem:decreasing2}
Every decreasing sequence of
 types of families of polynomial $\ell$-tuples is stationary.
\end{lemma}
\subsection{Nice families and the van der Corput operation}
 In this subsection we  define a class of families of $\ell$-tuples of polynomial
that we are going to work with in the sequel, and an important operation that preserves
such families and reduces their type.

\subsubsection{Nice families}
Let $\P_1=(p_{1,1},\ldots,p_{1,m})$, $\ldots$, $\P_1=(p_{\ell,1},\ldots,p_{\ell,m})$.
\begin{definition*}
We call the ordered  family of polynomial $\ell$-tuples $(\P_1,\ldots,\P_\ell)$ \emph{nice} if
\begin{enumerate}
\item\label{it:nice21}
 $\deg(p_{1,1})\geq \deg(p_{1,j})$ for $j=1,\ldots,m$\ ;
\medskip

\item\label{it:nice22}
$\deg(p_{1,1})>\deg(p_{i,j})$ for $i=2,\ldots,\ell$, $j=1,\ldots,m$\ ;

\medskip
\item\label{it:nice23}
$\deg(p_{1,1}-p_{1,j})>\deg(p_{i,1}-p_{i,j})$ for $i=2,\ldots,\ell$,
$j=2,\ldots,m$.
\end{enumerate}
(Notice  that a consequence of~\eqref{it:nice23} is that
$p_{1,1}-p_{1,j}$ is not constant for   $j= 2,\ldots,m$.)
\end{definition*}

The type of a nice family of degree $1$ has only one non-zero entry, namely
$w_{1,1}$.

\subsubsection{The van der Corput operation}
Given a family $\P=\big(p_1,\ldots,p_m\big)$,   $p\in \Z[t]$, and $h\in\N$, we define
$S_h\P$ and $\P-p$ as in Section~\ref{subsec:vdcoperation}.
Given a family of $\ell$-tuples of polynomials  $(\P_1,\ldots,\P_\ell)$,
 $(p_1,\cdots,p_\ell)\in (\P_1,\cdots\P_\ell)$, and $h\in\N$,
we define the following operation
 $$
(p_1,\ldots, p_\ell,h)\vdc(\P_1,\ldots, \P_\ell) =(\tilde{P}_{1,h},\ldots \tilde{P}_{\ell,h})^*
 $$
 where
 $$
  \tilde{P}_{i,h}=(S_h\P_i-p_i,\P_i-p_i).
 $$
for $i=1,\ldots,\ell$,  and  $^\ast$ is the operation that removes
all  constant $\ell$-tuples polynomials  from a given family of
 $\ell$-tuples polynomials.
 Notice that if $(\P_1,\ldots, \P_\ell)$ is a degree $d$ family containing $m$  polynomial $\ell$-tuples, then  for every $h\in\N$,
 the family
 $(p_1,\ldots, p_\ell,h)\vdc(\P_1,\ldots, \P_\ell)$ has degree at most $d$ and contains  at most
  $2m$ polynomial  $\ell$-tuples.

\subsection{Choosing a good vdC operation}
As in the case of two transformations,  our objective is starting with a nice family $(\P_1,\ldots, \P_\ell)$
to successively apply appropriate operations $(p_1,\ldots,
p_\ell,h)\vdc(\P_1,\ldots, \P_\ell)$
in order to arrive to  nice families of polynomial $\ell$-tuples  with  types
 that have  only non-zero entry the entry $w_{1,1}$. This case then can be treated easily using known  results that involve a  single transformation.

\begin{lemma}\label{L:reduceA'}
Let $(\P_1,\ldots, \P_\ell)$ be  a nice family with $\deg(p_{1,1})\geq 2$.

 Then there exists $(p_1,\ldots,p_\ell)\in (\P_1,\ldots,\P_\ell)$ such that for every large enough $h\in\N$
the family  $(p_1,\ldots,p_\ell,h)\vdc(\P_1,\ldots, \P_\ell)$ is nice and has strictly smaller type than
  $(\P_1,\ldots, \P_\ell)$.
\end{lemma}
\begin{proof}
We remind the reader that  we have
$\P_i=(p_{i,1},\ldots,p_{i,m})$ for $i=1,\ldots,\ell$.
For $(p_1,\ldots, p_\ell)\in (\P_1,\ldots, \P_\ell)$,
the  family $(p_1,\ldots,p_\ell,h)\vdc(\P_1,\ldots, \P_\ell)$  consists of vectors of polynomials that have the form
$$
(S_hp_{1,j}-p_1,\ldots, S_hq_{\ell,j}-p_\ell), \ j= 1,\ldots,m,
\  \text{ or } \
(p_{1,j}-p_1, \ldots, p_{\ell,j} -p_\ell).
$$

We choose $(p_1,\ldots, p_\ell)$ as follows:

  If $\P_\ell'$ is non-empty,
then we take $p_1=\cdots=p_{\ell-1}=0$ and $p_\ell$ to be a
 polynomial of smallest degree in $\P_\ell'$.
Then for every $h\in\N$, the first $\ell-1$ rows of the type will remain unchanged,
and the last row will get ``reduced'', leading to a smaller matrix type.
Similarly, if the families $\P_\ell', \P_{\ell-1}',\ldots, \P_{i-1}'$
are empty, and $P_{i}'$ is non-empty for some $2 \leq i \leq \ell+1$,
then we take $p_1=\cdots=p_{i-1}=0$ and $p_i$ to be a polynomial of smallest
degree in $\P_i'$.
Then for every $h\in\N$,
the first $i-1$ rows of the matrix type remain unchanged, and the
$i$-the row will get ``reduced'', leading to a smaller matrix type.

 Suppose now that the families $\P_\ell', \P_{\ell-1}',\ldots, \P_{2}'$ are empty.
If $\P_1$ consists of a single polynomial, namely $p_{1,1}$, then we
choose $(p_1,\ldots,p_\ell)=(p_{1,1},\ldots, p_{\ell,1})$ and the
result follows.
Therefore, we can assume that $\P_1$ contains some
polynomial other than $p_{1,1}$.
We consider two cases.
If $p\sim p_{1,1}$ for all $p\in \P_1$, then we choose
$(p_1,\ldots,p_\ell)=(p_{1,1},\ldots, p_{\ell,1})$.
Otherwise, we choose $(p_1,\ldots,p_\ell)\in (\P_1,\ldots,\P_\ell)$ with
$p_1\nsim p_{1,1}$, and $p_1$ is a polynomial in $\P_1$ with smallest degree
(such a choice exists since $p_{1,1}$ has the highest degree in
$\P_1$).
In all cases, for every $h\in\N$, the first row of the
matrix type of $(p_1,\ldots,p_\ell,h)\vdc(\P_1,\ldots, \P_\ell)$
is ``smaller'' than that of $(\P_1,\ldots, \P_\ell)$.

It remains to verify that for every large enough $h\in\N$ the family
 $(p_1,\ldots,p_\ell,h)\vdc(\P_1,\ldots, \P_\ell)$
 is nice. This part is identical with the one used in  Lemma~\ref{L:reduceA} and so we omit it.
\end{proof}

The proof of the next lemma is completely analogous to the proof of Lemma~\ref{L:k(d,m)} in the previous section
and so we omit it.
\begin{lemma}\label{L:k(d,m)'}
Let $(\P_1,\ldots,\P_\ell)$ be a nice family with degree $d\geq 2$ that  contains $m$ polynomial $\ell$-tuples.
Suppose that we successively apply the  $(p_1,\ldots,p_\ell,h)\vdc$ operation for appropriate choices of
$p_1,\ldots,p_\ell\in \Z[t]$ and $h\in\N$, as described in the previous lemma,
each time getting a nice family of $\ell$-tuples of polynomials with
strictly smaller type.

Then after a finite  number of  operations we get, for a good set of parameters,
   nice families of $\ell$-tuples of polynomials
of degree $1$.
Moreover,  the number of  operations needed can be bounded by
a function of $d$, $\ell$, and $m$.
\end{lemma}

\subsection{Proof of Proposition~\ref{P:CharB}} Using Lemma~\ref{L:reduceA'} and Lemma~\ref{L:k(d,m)'},
the rest of the proof of Proposition~\ref{P:CharB}  is completely analogous to the end of the proof of
Proposition~\ref{P:CharB2} given in  Section~\ref{SS:CharB2}  and so we omit it.

\section{Characteristic factors for the
lower degree iterates and proof of convergence}\label{S:Charlower}
In this section we prove Theorem~\ref{T:CharA} and then Theorem~\ref{T:Conv}.


\subsection{A simple example}
\label{subsec:simplexample2}
In order to explain our method, we continue with the example of
Section~\ref{subsec:simplexample}, studying the limiting behavior of the averages of
\begin{equation}
\label{eq:examplesuite}
 f_1(T_1^{n^2}x)\cdot f_2(T_2^nx).
\end{equation}
We have shown that these averages converge to $0$ in $L^2(\mu)$
whenever $f_1\perp \CZ_{2,T_1}$. We are therefore reduced to study
these averages under the additional hypothesis that $f_1$ is
measurable with respect to $\CZ_{2,T_1}$.

Using the approximation property of Proposition~\ref{L:ApprNil},
we further reduce matters to the case where, for $\mu$-almost every $x\in X$, the sequence
$(f_1(T^nx))_{n\in\N}$ is a $2$-step nilsequence. Therefore, the
sequence $(f_1(T^{n^2}x))_{n\in\N}$ is a $4$-step nilsequence.
We are left with studying the limiting behavior of  the averages of
$$
  u_n(x)\cdot f_2(T_2^nx),
$$
where $(u_n)_{n\in\N}$ is a uniformly bounded sequence of $\mu$-measurable functions,
such that $(u_n(x))_{n\in\N}$ is a $4$-step nilsequence for $\mu$-almost
every $x\in X$.

In this particular case, Corollary~\ref{cor:uniformnil} below suffices to show
that the averages converge to $0$ in $L^2(\mu)$ whenever
$f_2\perp\CZ_{4,T_2}$. (For more intricate averages we need more
elaborate results about weighted multiple averages.)

We are reduced to the case where $f_1$ is
measurable with respect to $\CZ_{2,T_1}$ and $f_2$ is measurable with
respect to $\CZ_{4,T_2}$. Applying Proposition~\ref{L:ApprNil} to
these two functions, we reduce matters to the case where, for $\mu$-almost every
$x\in X$, the sequences $(f_1(T_1^{n^2}x))_{n\in\N}$ and
$(f_2(T_2^nx))_{n\in\N}$ are finite step nilsequences. Therefore, for $\mu$-almost every
$x\in X$, the sequence~\eqref{eq:examplesuite} is a nilsequence and as a consequence
its averages converge.

  We introduce now the tools that we need to carry out the previous plan in our more general setup.

\subsection{Uniformity seminorms}
We follow~\cite{HKc}. Let $k\in\N$ be an integer. Let
$(a_n)_{n\in\Z}$ be a bounded  sequence of real numbers and
$\bI=(I_N)_{N\in\N}$ be a sequence of intervals whose lengths
$|I_N|$ tend to infinity. We say that this sequence of intervals is
\emph{$k$-adapted} to the sequence $(a_n)$, if for every
$\uh=(h_1,\cdots,h_k)\in\N^k$, the limit
$$
c_\uh(\bI,a):=\lim_{N\to\infty}  \frac{1}{|I_N|} \sum_{n\in I_N}
\prod_{\epsilon\in\{0,1\}^k}a_{n+h_1\epsilon_1+\cdots+h_k\epsilon_k}
$$
exists\footnote{In~\cite{HKc} it is assumed that the limit exists for
$\uh\in\Z^k$ but this does not change anything in the proofs.}.
 Clearly, every sequence of intervals whose lengths tend to
infinity admits a subsequence which is adapted to the sequence $(a_n)$.

Suppose that $\bI=(I_N)_{N\in\N}$  is $k$-adapted to  $(a_n)_{n\in\Z}$.
We define
$$
 \nnorm a_{\bI,k}:=\Bigl(\lim_{H\to+\infty}\frac
1{H^k}\sum_{1\leq h_1,\cdots,h_k\leq H} c_{\uh}(\bI,a)
\Bigr)^{1/2^k}.
$$
Indeed, by Proposition~2.2 of~\cite{HKc}, the above limit exists and is
non-negative.
\begin{lemma}
\label{lem:subsec}
Let $(X,\CX,\mu,T)$ be a system, $f\in L^\infty(\mu)$, and
$\bI=(I_N)_{N\in\N}$ be a sequence of intervals whose lengths tend to
infinity. Suppose that $f\bot \CZ_{k-1,\mu}$ for some $k\geq 2$.

Then the sequence $\bI$ admits a subsequence $\bI'=(I'_N)_{N\in\N}$
such that, for $\mu$-almost every $x\in X$,
$\bI'$ is $k$-adapted to the sequence $(f(T^nx))_{n\in \N}$ and
$\nnorm{f(T^nx)}_{\bI',k}=0$.
\end{lemma}
\begin{proof}
Let $\mu=\int\mu_x\,d\mu(x)$ be the ergodic decomposition of $\mu$.
For $x\in X$, we write $a(x)=(a_n(x))_{n\in\N}$ for the sequence defined
by $a_n(x)=f(T^nx)$.

By the Ergodic Theorem, for  every $\uh=(h_1,\cdots,h_k)\in\N^k$,
the averages
$$
 \frac 1{|I_N|}\sum_{n\in I_N}\prod_{\epsilon\in\{0,1\}^k}
a_{n+h_1\epsilon_1+\cdots+h_k\epsilon_k}(x)
$$
converge in $L^2(\mu)$. As a consequence,  a subsequence of this sequence of
averages converges $\mu$-almost everywhere. This subsequence depends on
the parameter $\uh$, but since there are only countably many such
parameters, by a diagonal argument we can find a subsequence $\bI'
=(I'_N)_{N\in\N}$ such that for $\mu$ almost every $x\in X$ the limit
\begin{equation}
\label{eq:avIprime}
 c_{\uh}(\bI',a(x))=\lim_{N\to+\infty}
\frac 1{|I'_N|}\sum_{n\in I'_N}\prod_{\epsilon\in\{0,1\}^k}
a_{n+h_1\epsilon_1+\cdots+h_k\epsilon_k}(x)
\end{equation}
exists for every choice of $\uh=(h_1,\cdots,h_k)\in \N^k$. This
means that, for $\mu$-almost every $x\in X$, the sequence of
intervals $\bI'$ is $k$-adapted to the
 sequence $(a_n(x))_{n\in \N}$.

Furthermore, by the Ergodic Theorem, for every $\uh\in\N^k$ the
averages on the right hand side of~\eqref{eq:avIprime}
converge in $L^2(\mu)$ to
$$
\E_\mu\Bigl(\prod_{\epsilon\in \{0,1\}^k}
 T^{h_1\epsilon_1+\cdots+ h_k\epsilon_k}f
\Big|\CI(T)\Bigr)(x)=
 \int
\prod_{\epsilon\in \{0,1\}^k}
 T^{h_1\epsilon_1+\cdots+h_k\epsilon_k}f\;d\mu_x.
$$
 Therefore,
 for $\mu$-almost every $x\in X$, we have
$$
 c_{\uh}(\bI',a(x))=\int
\prod_{\epsilon\in \{0,1\}^k}
 T^{h_1\epsilon_1+\cdots +h_k\epsilon_k}f\;d\mu_x.
$$
Taking the average in $\uh$, using the definition of $\CD_kf$
(Section~\ref{subsec:dual}),
and~\eqref{eq:CDff}, we get  for $\mu$-almost every $x\in X$ that
$$
 \nnorm {a(x)}_{\bI',k}=\nnorm f_{k,\mu_x}.
$$
Since by hypothesis $\E_\mu(f|\CZ_{k-1})=0$, by~\eqref{E:DefZ_l}
we have $\nnorm f_{k,\mu}=0$, and as a consequence $\nnorm f_{k,\mu_x}=0$ for
$\mu$-almost every $x\in X$ by~\eqref{E:nonergodic}. This completes the proof.
\end{proof}

We are also going to use the following result:
\begin{theorem}[\cite{HKc}, Corollary 2.14]\label{T:HKdirect}
Let  $(a_n)_{n\in\N}$ be a bounded sequence  of real numbers, and
$\bI=(I_N)_{N\in\N}$ be a sequence of intervals that is $k$-adapted to this
sequence for some $k\geq 2$.
Suppose that  $\nnorm{a_n}_{{\bf I},k}=0$.

Then for every bounded $(k-1)$-step nilsequence $u_n$  we have
$$
\lim_{N\to\infty}  \frac{1}{|I_N|}\sum_{n\in I_N} a_n u_n=0.
$$
\end{theorem}

Combining the results of this section, we can now prove:
\begin{corollary}
\label{cor:uniformnil}
Let $(X,\CX,\mu,T)$ be a system and $f\in L^\infty(\mu)$.
Let $(u_n(x))_{n\in\N}$ be a uniformly bounded sequence of
$\mu$-measurable functions such that,
 for $\mu$-almost every $x\in X$, the sequence $(u_n(x))_{n\in\N}$ is a
$k$-step nilsequence for some $k\geq 1$. Suppose that $f\bot \cZ_{k,T}$.

 Then the averages
$$
\frac{1}{N-M}\sum_{n=M}^{N-1} f(T^nx)\cdot u_n(x)
$$
converge to $0$ in $L^2(\mu)$.
\end{corollary}
\begin{proof}
It suffices to prove that every sequence of intervals
$\bI=(I_N)_{n\in\N}$ whose lengths tend to infinity admits a
subsequence $\bI'=(I'_N)_{n\in\N}$ such that
\begin{equation}\label{E:zew}
 \frac 1{|I'_N|}\sum_{n\in I'_N} f(T^nx)\cdot u_n(x)\to 0\text{ in
}L^2(\mu).
\end{equation}
Let $\bI'$ be given by Lemma~\ref{lem:subsec} (with $k$ in place of $k-1$). For $\mu$-almost every $x\in X$ we have
$\norm{(f(T^nx))_{n\in \N}}_{\bI',k+1}=0$. Theorem~\ref{T:HKdirect} gives that the averages
in \eqref{E:zew} converge to $0$ pointwise and the asserted convergence to $0$ in $L^2(\mu)$ follows from
 the bounded convergence theorem. This completes the proof.
\end{proof}

\subsection{Some weighted averages}We are going to prove Theorem~\ref{T:CharA} by induction on the number of
transformations involved.  The next result is going to help us carry out the induction step.
\begin{proposition}\label{prop:weighted}
Let $(X,\X,\mu,T_1,\cdots,T_\ell)$ be a  system and
$f_1,\ldots,f_\ell\in L^\infty(\mu)$. Let
 $p_1,\ldots,p_\ell\in \Z[t]$ be polynomials with distinct degrees and highest degree $d=\deg (p_1)$.
Let $(u_n(x))_{n\in\N}$ be a uniformly bounded sequence of
$\mu$-measurable functions such that,
 for $\mu$-almost every $x\in X$, the sequence $(u_n(x))_{n\in\N}$ is an
$s$-step nilsequence for some $s\geq 1$.

Then there  exists $k=k(d,\ell,s)\in \N$ such that: If
    $f_1\perp \cZ_{k,T_1}$, then the averages
$$
\frac{1}{N-M}\sum_{n=M}^{N-1}  f_1(T_1^{p_1(n)}x)\cdot \ldots\cdot f_\ell(T_\ell^{p_\ell(n)}x)\cdot u_n(x)
$$
 converge to $0$ in  $L^2(\mu)$.
\end{proposition}

\begin{proof}
First suppose that  $\deg (p_1)=1$. Then the polynomials $p_2,\ldots,p_\ell$  are all
 constant. The polynomial $p_1$ has the form $p_1(n)=an+b$ for some
 integers $a,b$ with $a\neq 0$. Applying Corollary~\ref{cor:uniformnil} for  $T_1^bf_1$ in place of
  $f$ and $T_1^a$ in place of  $T$, and using \eqref{E:powers} we get
 the announced result
with $k=s+1$.

Therefore, we can assume that $\deg (p_1)\geq 2$. The strategy  of the
proof is the same as in Corollary~\ref{cor:uniformnil}, but instead of the Ergodic Theorem used in the
proof of Lemma~\ref{lem:subsec}, we use Proposition~\ref{P:CharB}.

We assume that $f_1\perp \cZ_{k,T_1}$, where $k$ is the integer
$k(d,\ell,2^s\ell )$ given by Proposition~\ref{P:CharB}.
In order to prove the announced convergence to $0$, it suffices to show that
every sequence of intervals $\bI=(I_N)_{N\in\N}$ admits a subsequence
$\bI'=(I'_N)_{N\in \N}$ such that
\begin{equation}
\label{eq:Inprime}
 \frac 1{|I'_N|}\sum_{n\in I'_N}
 f_1(T_1^{p_1(n)}x)\cdot \ldots\cdot f_\ell(T_\ell^{p_\ell(n)}x)\cdot u_n(x)\
\text{ converges to $0$ in $L^2(\mu)$.}
\end{equation}

We let $m=2^{s}$, and for $x\in X$, let $a(x)=(a_n(x))_{n\in\Z}$ be
the sequence given by
$$
a_n(x)=f_1(T_1^{p_1(n)}x) \cdot \ldots \cdot
f_{\ell}(T_{\ell}^{p_{\ell}(n)}x).
$$

For $r_1,\cdots,r_m\in\Z$, we study the averages
$$
\frac 1{|I_N|}\sum_{n\in I_N} a_{n+r_1}(x)\cdots a_{n+r_m}(x).
$$
Consider the following $\ell$ ordered families   of
polynomials, each consisting of $\ell m$ polynomials:
\begin{align*}
\P_1&= \bigl(p_1(n+r_1),\dots,p_1(n+r_m), 0, \dots,0, \hbox to 1.5cm{\dotfill}, 0,\dots,0\bigr)\\
\P_2&=\bigl( 0, \dots,0,p_2(n+r_1),\dots,p_2(n+r_m),\hbox to
1.5cm{\dotfill},  0,\dots,0\bigr)\\
 \dots&\hbox to 9.7cm{\dotfill}\\
\P_\ell&=\bigl(0,\dots,0,0,\dots,0, \hbox to
1.5cm{\dotfill},p_\ell(n+r_1),\dots,p_\ell(n+r_m)\bigr)
\end{align*}
Using that $\deg(p_1)\geq 2$ and $\deg(p_i)<\deg(p_1)$ for
$i=2,\ldots,\ell$, it is easy to check that this family is nice
except if $r_1\in\{r_2,\cdots,r_m\}$.

 Using Proposition~\ref{P:CharB} (with  $k=k(d,\ell,2^s\ell)$)
we have that the averages
$$
 \frac 1{|I_N|}\sum_{n\in I_N} a_{n+r_1}(x)\cdot\ldots\cdot a_{n+r_m}(x)
$$
 converge to $0$ in $L^2(\mu)$ for every  $r_1,\ldots,r_{m}\in\Z$
with $r_1\notin\{r_2,\ldots,r_m\}$.
As in the proof of
Lemma~\ref{lem:subsec}, there exists a subsequence
$\bI'=(I'_N)_{N\in\N}$ of the  sequence of intervals $\bI$ such that
$$
 \frac 1{|I'_N|}\sum_{n\in I'_N}a_{n+r_1}(x)\cdot\ldots\cdot
a_{n+r_m}(x)\to 0 \ \ \mu\text{-almost everywhere}
$$
for all choices of $r_1,\ldots,r_{m}\in\Z$ with $r_1\notin\{r_2,\ldots,r_m\}$.

In particular, for every $h_1,\cdots,h_s\in\N$, we have
$$
 \frac 1{|I'_N|}\sum_{n\in
I'_N}\prod_{\epsilon\in\{0,1\}^s}a_{n+\epsilon_1h_1+\cdots+\epsilon_sh_s}(x)
\to 0 \ \ \mu\text{-almost everywhere}.
$$
To see this,  apply the previous convergence property when  $\{
r_1,\cdots,r_m\}$ is equal to the set $\big\{\epsilon_1h_1 + \cdots
+ \epsilon_sh_s, \epsilon_i \in \{0,1\}\big\}$ and  $r_1=0$.

 As a consequence, for $\mu$-almost every $x\in X$,  the sequence $\bI'$ of intervals
is $s$-adapted to the sequence $a(x)$,   and $c_{\uh}(\bI',a(x))=0$ for
every $\uh\in\N^s$. Therefore,
$\nnorm{a(x)}_{\bI',s}=0$ for $\mu$-almost every $x\in X$.
By Theorem~\ref{T:HKdirect}, we have
$$
\frac 1{|I'_N|}\sum_{n\in I'_N}a_n(x)\cdot u_n(x)\to 0  \ \ \mu\text{-almost everywhere}
$$
 and~\eqref{eq:Inprime} is proved. This completes the proof.
\end{proof}

\subsection{Proof of Theorem~\ref{T:CharA} }\label{SS:ProofA}
We are now ready to prove Theorem~\ref{T:CharA}. It is a special case (take $u_n$ to be constant)  of the following result:
\begin{theorem}
\label{th:charcnilseq}
Let $(X,\CX,\mu,T_1,\cdots,T_\ell)$ be a system and
$f_1,\cdots,f_\ell\in L^\infty(\mu)$. Let $p_1,\cdots,p_\ell$ be
polynomials with distinct degrees and maximum degree $d$. Let
$(u_n(x))_{n\in\N}$ be a uniformly bounded sequence of measurable
functions on $X$ such that, for $\mu$-almost every $x\in X$, the
sequence $(u_n(x))_{n\in\N}$ is an $s$-step nilsequence.

 Then there exists
$k=k(d,\ell,s)$ with the following property: If $f_i\perp\CZ_{k,T_i}$
for some $i\in\{1,\ldots,\ell\}$, then the
averages
\begin{equation}
\label{eq:weighted}
 \frac{1}{N-M}\sum_{n=M}^{N-1}f_1(T^{p_1(n)}x)\cdot\ldots\cdot f_\ell(T^{p_\ell(n)}x)\cdot u_n(x)
\end{equation}
converge to $0$ in $L^2(\mu)$.
\end{theorem}
\begin{proof}
The proof goes by induction on the number $\ell$ of transformations.
For $\ell=1$, the result is the case $\ell=1$ of
Proposition~\ref{prop:weighted}. We take $\ell\geq 2$, assume that
the results holds for $\ell-1$ transformations, and we are going to prove that it
holds for $\ell$ transformations.

Without loss of generality we can assume that
$\deg(p_1)=d>\deg(p_i)$ for $2\leq i\leq\ell$.
By  Proposition~\ref{prop:weighted}, there exists $k_0=k_0(d,\ell,s)$ such that,
if $f_1\perp\CZ_{k_0,T_1}$, then the averages~\eqref{eq:weighted}
converge to $0$ in $L^2(\mu)$.
Therefore we can restrict ourselves to the case  where
\smallskip

\centerline{\emph{the function $f_1$ is measurable
with respect to $\CZ_{k_0,T_1}$.}}

\smallskip

By Proposition~\ref{L:ApprNil}, for every $\varepsilon >0$, there exists
$\tilde{f_1}\in L^\infty(\mu)$, measurable with respect to
$\CZ_{k_0,T_1}$, with
$\norm{f_1-\tilde{f_1}}_{L^2(\mu)}<\varepsilon$, and such that
$(\tilde{f_1}(T_1^nx))_{n\in\N}$ is a $k_0$-step nilsequence for
$\mu$-almost every $x\in X$.
By density, it suffices to prove the result under the additional
hypothesis that
\smallskip

\centerline{\emph{$(f_1(T_1^nx))_{n\in\N}$ is a $k_0$-step nilsequence for
$\mu$-almost every $x\in X$.}}

\smallskip
Then for $\mu$-almost every $x\in X$, the sequence
$(f_1(T_1^{p_1(n)}x))_{n\in\N}$ is a $(dk_0)$-step nilsequence. The
sequence
$(f_1(T_1^{p_1(n)}x)\cdot u_n(x))_{n\in\N}$ is the product of two
$k$-step nilsequences where $k=\max(dk_0,s)$  and thus it is a $k$-step
nilsequence. Therefore, the announced result follows from the induction
hypothesis. This completes the induction and the proof.
\end{proof}

\subsection{Proof of Theorem~\ref{T:Conv}}\label{SS:proof}
Let $(X,\X,\mu,T_1,\cdots,T_\ell)$ be a system and
$f_1,\ldots,f_\ell\in L^\infty(\mu)$. We assume  that the
polynomials $p_1,\ldots,p_\ell\in \Z[t]$ have distinct degrees and
we want to show that the averages
\begin{equation}
\label{eq:conv}
\frac{1}{N-M}\sum_{n=M}^{N-1} f_1(T_1^{p_1(n)}x)\cdot \ldots\cdot f_\ell(T_\ell^{p_\ell(n)}x)
\end{equation}
converge in $L^2(\mu)$.

By Theorem~\ref{T:CharA}, there exists $k\in\N$ such that the
averages~\eqref{eq:conv} converge to $0$ whenever
$f_i\perp\CZ_{k,T_i}$ for some  $i\in\{1,\cdots,\ell\}$. Therefore,
we can assume that for $i=1,\ldots,\ell$, the function $f_i$ is
measurable with respect to $\CZ_{k,T_i}$.

By Proposition~\ref{L:ApprNil}, for every $\varepsilon >0$, and for
$i=1,\cdots,\ell$, there exists  a function $\tilde{f_i}\in
L^\infty(\mu)$, measurable with respect to $\CZ_{k,T_i}$, with
$\norm{f_i-\tilde{f_i}}_{L^2(\mu)}<\varepsilon$, and such that
$(\tilde{f_i}(T^nx))_{n\in\N}$ is a $k$-step nilsequence for
$\mu$-almost every $x\in X$.

By density we can therefore assume that, for  $i=1,\ldots,\ell$, and for $\mu$-almost every
$x\in X$, $(f_i(T^nx))_{n\in\N}$ is a
$k$-step nilsequence and as a consequence $(f_i(T^{p_i(n)}x))_{n\in\N}$ is a
$(dk)$-step nilsequence. Then
for $\mu$-almost every $x\in X$, the average \eqref{eq:conv} is an average of a $(dk)$-step nilsequence,
and therefore it converges by \cite{L}. This completes the proof. \qed

\section{Lower bounds for powers}\label{S:lower}

In this section we are going to prove Theorem~\ref{T:LowerBounds}.

We remark that a consequence of
Theorem~\ref{T:Conv} is that all the limits of multiple ergodic averages
mentioned in this section exist (in $L^2(\mu)$).  As a result,  we are allowed to write $\lim_{N-M\to\infty}$,
where $\limsup_{N-M\to\infty}$ should have been used.

We start with some background material.

\subsection{Equidistribution properties on nilmanifolds}\label{SS:equidistribution}
We summarize some notions and results  that will be needed later.

\subsubsection*{Polynomial sequences.}
Let $G$ be a nilpotent Lie group.
Let  $X=G/\Gamma$ be a nilmanifold, where $\Gamma$ is a discrete
cocompact  subgroup of $G$.
Recall that for $a\in G$ we write $T_a\colon X\to X$ for the
\emph{translation} $x\mapsto ax$.

If $a_1,\ldots, a_\ell\in G$, and
$p_1, \ldots, p_\ell \in \Z[t]$, then a sequence of the form
 $g(n) =a_1^{p_1(n)}a_2^{p_2(n)}\cdots$ $a_\ell^{p_\ell(n)}$
  is called a {\it polynomial sequence in $G$}.
 If $x\in X$ and $(g(n))_{n\in\N}$ is a polynomial
  sequence in $G$, then the sequence
 $(g(n)x)_{n\in\N}$
  is called a {\it polynomial sequence in $X$}.

\subsubsection*{Sub-nilmanifolds.}
If $H$ is a closed subgroup of $G$ and $x\in X$, then $Hx$ may not be a  closed
subset of $X$ (for example, take $X=\R/\Z$, $x=\Z$, and
$H=\{k\sqrt{2}\colon k\in \Z\}$), but if it is closed, then the compact set $Hx$  can be given the structure of a nilmanifold (\cite{L}).
More precisely, if $x=g\Gamma$, then  $Hx$ is closed if and only if
$\Delta=H\cap g\Gamma g^{-1}$ is cocompact in $H$. In this case
 $Hx\simeq
H/\Delta$, and $h\mapsto hg\Gamma $  induces the isomorphism from $H/\Delta$ onto $Hx$. We call any such set $Hx$ a {\it
sub-nilmanifold} of $X$.

\subsubsection*{Equidistribution.} We say that the sequence $(g(n)x)_{n\in\N}$, with
values in a nilmanifold $X$, is \emph{equidistributed} (or \emph{well distributed})
in a sub-nilmanifold $Y$ of $X$,
 if for every $F\in \CC(X)$ we have
$$
\lim_{N-M\to\infty}\frac{1}{N-M}\sum_{n=M}^{N-1} F(g(n)x)=\int F\ dm_{Y}
$$
where $m_{Y}$ denotes the normalized Haar measure on $Y$.

For typographical reasons, we use the following
notation:
\begin{notation} If $E$ is a subset of $X$, we denote by  $\cl_X (E)$
the closure of $E$  in $X$.
\end{notation}
A fact that we are going to use repeatedly is that polynomial sequences are equidistributed
in their orbit closure. More precisely:
\begin{theorem}[{\bf  \cite{L}}]\label{T:L1}
Let  $X = G/\gG$ be a  nilmanifold,  $(g(n)x)_{n\in\N}$ be a polynomial
sequence in $X$, and  $Y=\cl_X\{g(n)x, n\in\N\}$.
\begin{enumerate}
\item\label{it:TL1ii}
 There exists $r\in \N$ such that  the
sequence $(g(rn)x)_{n\in\N}$ is equidistributed on some connected component of $Y$.
\item\label{it:TL1i}
If $Y$ is connected, then $Y$ is a sub-nilmanifold of $X$, and for every $r\in \N$ the sequence
$(g(rn)x)_{n\in\N}$ is equidistributed on $Y$.
\end{enumerate}
\end{theorem}

\subsubsection*{Ergodic elements.}
An element  $a\in G$ is \emph{ergodic},
or \emph{acts ergodically on $X$},  if the sequence $(a^n\Gamma)_{n\in\N}$ is dense
in $X$.

Suppose that   $a\in G$ acts ergodically on $X$.
 Then  for
every $x\in X$ the sequence $(a^nx)_{n\in\N}$ is equidistributed in $X$.
  If  $X$ is assumed to be connected,  then
 for every $r\in \N$ the element  $a^r$ also acts ergodically on $X$ (this follows from part (iii) of Theorem~\ref{T:L1}).
For general nilmanifolds $X$ we can easily  deduce  the following result (with $X_0$ we denote the connected component of the element $\Gamma$): There exists $r_0\in\N$ such that the nilmanifold $X$ is the disjoint union of the sub-nilmanifolds $X_i=a^iX_0$, $i=1,\ldots,r_0$, and $a^{r}$ acts ergodically on each $X_i$ for every $r\in r_0\N$.

\subsubsection*{The affine torus.} If $X=G/\Gamma$ is a connected nilmanifold, the {\em affine torus} of
$X$ is defined to be the homogeneous space $A=G/([G_0,G_0]\Gamma)$, where by $G_0$ we denote the connected component
  of the identity element in $G$.
The homogeneous space $A$ can be smoothly identified in a natural way with
the nilmanifold $G_0/([G_0,G_0](\Gamma\cap G_0))$,
which is a
finite dimensional torus, say $\T^m$ for some $m\in\N$.
It is known (\cite{FK1}) that, under this identification, $G$ acts on
$A$ by unipotent affine transformations.
This means that every  $T_g\colon \T^m\to \T^m$
has the form
$Tx =  Sx+b$, for some unipotent homomorphism $S$ of $\T^m$
and $b\in \T^m$.

\subsubsection*{Equidistribution criterion.}
If $X=G/\Gamma$ is a nilmanifold,
then $X$ is connected if and only if $G=G_0\Gamma$.
In the sequel we need to establish
  some
equidistribution properties of polynomial sequences on
nilmanifolds. The next criterion  is going to  simplify our task:

\begin{theorem}[{\bf \cite{L}}]\label{T:L2}
Let  $X = G/\gG$ be a connected nilmanifold,  $(g(n))_{n\in\N}$ be a polynomial
sequence in $G$, and $x\in X$. Let   $A=G/([G_0,G_0]\Gamma)$  be the affine torus of $X$
and  $\pi_A\colon
X\to A$ be the natural projection.

Then the sequence
$(g(n)x)_{n\in\mathbb{N}}$ is equidistributed in $X$ if and only if
the sequence $(g(n)\pi_A(x))_{n\in\mathbb{N}}$ is equidistributed in $A$.
\end{theorem}

\subsection{An  example} In order to explain the  strategy of the proof of Theorem~\ref{T:LowerBounds} we
use an example.
 Our goal  is to show that for a given system $(X,\X,\mu,T_1,T_2)$, and set $A\in \X$, for
 every $\varepsilon>0$,  we have
 $$
 \mu(A\cap T_1^{-n}A\cap T_2^{-n^2}A)\geq \mu(A)^3 -\varepsilon
 $$
for a set of $n\in\N$ that has bounded gaps.

After some manipulations that are explained in Section~\ref{SS:start}, we are left with showing that
 if $f_1\bot \krat (T_1)$ or $f_2\bot \krat(T_2)$, then the averages of
\begin{equation}\label{E:simple}
 f_1(T_1^nx)\cdot f_2(T_2^{n^2}x)
 \end{equation}
 converge to $0$ in $L^2(\mu)$. In fact, we are only going to be able to prove a somewhat more technical variation
  of this property (see Proposition~\ref{P:OptimalChar}), but the exact details are not  important at this point.

By Theorem~\ref{T:CharA} we can assume that the function $f_1$ is $\cZ_{k,T_1}$-measurable and the function $f_2$ is $\cZ_{k,T_2}$-measurable for some $k\in \N$. For convenience, we also assume that the transformation $T_1$ is totally  ergodic (meaning $T_1^r$ is ergodic for every $r\in \N$). In this case,
using  Theorem~\ref{T:Structure} and an approximation argument, we can further reduce matters to the case where
$X$ is a connected nilmanifold,  $\mu=m_X$,  and $T_1=T_a$ is an ergodic translation on $X$. The assumption that
$X$ is connected is important, and is a consequence of our simplifying assumption that the transformation $T_1$ is totally ergodic.
 Also, by Proposition~\ref{L:ApprNil}, we can assume
that for $m_X$-almost every $x\in X$ the sequence $u_n(x)=f_2(T_2^{n}x)$ is a finite step nilsequence.

After doing all these maneuvers our new goal becomes to establish the following result:

{\bf (a)} Let $X$ be a connected nilmanifold, $a$ be an ergodic translation of $X$, and $\int f_1\ dm_X=0$.
Let  $(u_n)_{n\in\N}$ be a uniformly bounded sequence of measurable functions such that $(u_n(x))_{n\in\N}$ is a nilsequence
for $m_X$-almost every $x\in X$. Then the averages of
$$
 f_1(a^nx)\cdot u_{n^2}(x)
 $$
 converge to $0$ in $L^2(m_X)$. (The conclusion fails if $X$ is not connected.)

  It is easy to see that (a)   follows from the following  result:

{\bf (a)'} Let $X$ be a connected nilmanifold  and $a$ be an ergodic translation of $X$. Let
  $Y$ be a nilmanifold and $b$ be  an ergodic translation of $Y$.  Then for $m_X$-almost
  every $x\in X$ we have: for every nilmanifold $Y$, every ergodic translation $b$ of $Y$, and   every $y\in Y$, the sequence
  $$
  (a^nx,b^{n^2}y)
  $$
  is equidistributed on the nilmanifold $X\times Y$.

  We prove a variation of this result that suffices for our purposes in Lemma~\ref{L:NilEqui}.
  This is the heart of our argument, and  we prove it by (i)  showing that it suffices to verify
  the announced equidistribution property when each translation  $a$ and $b$ is given by an
  ergodic unipotent affine
  transformation on some finite dimensional torus, and then (ii) verify the announced equidistribution
   property for affine transformations
   using direct computations (see Lemma~\ref{L:affine}). It is in this second step that we make crucial
  use of the special structure of our polynomial iterates; our argument does not quite work for some
  other distinct degree polynomials  iterates like $n$ and $n^2+n$. The key observation is that since all the coordinates of the sequence $(a^nx)$ have non-trivial linear part, and those of  $(b^{n^2}y)$ have trivial linear part, for typical values of $x\in X$, it is impossible for the coordinates of the sequences $(a^nx)$ and $(b^{n^2}y)$ to ``conspire'' and complicate the equidistribution properties of the sequence $(a^nx,b^{n^2}y)$.

If the transformation $T_1$ is ergodic but not totally ergodic, then further technical issues arise, but  they are not hard to overcome. If  $T_1$ is not ergodic, then it is possible to use its ergodic decomposition, and the previously established ergodic result per ergodic component, to deduce the result for $T_1$.
Finally, if  $f_2\bot \krat(T_2)$, we first use the previously established result to reduce matters to the case where the function $f_1$ is $\krat(T_1)$-measurable, and then it becomes an easy matter to show that the averages of  \eqref{E:simple} converge to $0$ in $L^2(\mu)$.

\subsection{Proof of Theorem~\ref{T:LowerBounds} modulo a convergence result}\label{SS:start}
We are going to derive Theorem~\ref{T:LowerBounds} from the following result (that will be proved in the next subsection):
\begin{proposition}\label{P:OptimalChar}
Let $(X,\X,\mu,T_1,\cdots,T_\ell)$ be a system. Let
$d_1,\ldots,d_\ell\in \N$ be distinct and $f_1,\ldots,f_\ell\in
L^\infty(\mu)$. Suppose  that $f_i\perp\krat(T_i)$ for some
$i=1,\ldots,\ell$.

Then  for every $\varepsilon>0$, there exists $r_0\in \N$, such that  for every $r\in r_0\N$, we have
\begin{equation}\label{E:approx}
\lim_{N-M\to\infty}\norm{\frac{1}{N-M}\sum_{n=M}^{N-1} f_1(T_1^{(rn)^{d_1}}x)\cdot \ldots\cdot f_\ell(T_\ell^{(rn)^{d_\ell}}x)}_{L^2(\mu)}
\leq \varepsilon.
\end{equation}
\end{proposition}
(The existence of the limit is given by Theorem~\ref{T:Conv}.)
\begin{remark}
The conclusion should hold with $r_0=1$ and $\varepsilon=0$, but we currently do not see how to show this.
\end{remark}

We are also going to need the next inequality, it is proved by an appropriate application of
H\"older's inequality:

\begin{lemma}[\cite{Chu}]\label{L:LowerBound}
Let $\ell\in \N$, $(X,\mathcal{X},\mu)$ be a probability space,
$\X_1, \X_2, \ldots, \X_\ell$  be sub-$\sigma$-algebras of $\X$,
and  $f\in L^{\infty}(\mu)$ be non-negative.

Then
$$
\int f\cdot \E(f|\X_1)\cdot  \E(f|\X_2)\cdot \ldots\cdot
\E(f|\X_\ell) \,
 d\mu\geq \Big(\int f \ d\mu\Big)^{\ell+1}.
$$
\end{lemma}

\begin{proof}[Proof of Theorem~\ref{T:LowerBounds} assuming Proposition~\ref{P:OptimalChar}]
Let $\varepsilon>0$. It suffices to show that there exists $r\in \N$ such that
\begin{equation}\label{E:key}
\lim_{N-M\to \infty}\frac{1}{N-M}\sum_{n=M}^{N-1} \mu(A\cap T_1^{(rn)^{d_1}}A\cap \cdots\cap T_\ell^{(rn)^{d_\ell}}A)\geq \mu(A)^{\ell+1}-2\varepsilon.
\end{equation}

   First we use  Proposition~\ref{P:OptimalChar} to choose  $r_0\in \N$ so that  for every $r\in r_0\N$
   we have the estimate \eqref{E:approx} with $\varepsilon/2^\ell$ in place of $\varepsilon$. Next we choose a multiple $r$ of $r_0$ such that
for $i=1,\ldots,\ell$ we have
\begin{equation}\label{E:K_r}
\norm{\E({\bf 1}_A|\K_r(T_i))-\E({\bf 1}_A|\krat(T_i))}_{L^2(\mu)}\leq \frac{\varepsilon}{\ell}.
\end{equation}
We claim that for this choice of $r$ equation~\eqref{E:key} holds.
Indeed  by~\eqref{E:approx} (with $\varepsilon/2^\ell$ in place of $\varepsilon$)
we have that the limit in~\eqref{E:key} is $\varepsilon$-close to the limit of the averages of
\begin{equation}\label{E:key'}
\int {\bf 1}_A\cdot  T_1^{(rn)^{d_1}}\E({\bf 1}_A|\krat(T_1))\cdot
\ldots\cdot T_\ell^{(rn)^{d_\ell}}\E({\bf 1}_A|\krat(T_\ell))\, d\mu.
\end{equation}
Using \eqref{E:K_r} we easily conclude that the limit in \eqref{E:key'} is $\varepsilon$ close to the limit of the  averages of
$$
\int
{\bf 1}_A\cdot  T_1^{(rn)^{d_1}}\E({\bf 1}_A|\K_r(T_1))\cdot \ldots\cdot T_\ell^{(rn)^{d_\ell}}\E({\bf 1}_A|\K_r(T_\ell))\ d\mu.
$$
Since $T^rf=f$ for $\K_r(T)$-measurable functions $f$, the last expression is equal to
$$
\int {\bf 1}_A\cdot  \E({\bf 1}_A|\K_r(T_1))\cdot \ldots\cdot \E({\bf 1}_A|\K_r(T_\ell))\ d\mu.
$$
By Lemma~\ref{L:LowerBound}, the last integral is greater or equal than $\mu(A)^{\ell+1}$. It follows that \eqref{E:key} holds and  the proof is complete.
\end{proof}

\subsection{Some equidistribution results}
In the next subsection we prove Proposition~\ref{P:OptimalChar}.
A crucial step in the proof is  an
equidistribution result on nilmanifolds that we  prove in this subsection. We start
with a lemma.

\begin{lemma} \label{L:affine}
Let $d,m_1,m_2\in \N$ and
 $T\colon \T^{m_1}\to \T^{m_1}$ be an ergodic unipotent
affine transformation.
For  $i=1,\ldots,m_2$, let $u_i\in \R[t]$  be a polynomial divisible by $t^{d+1}$.
Suppose that  the sequence $(u(n))_{n\in\N}$, with values in $\T^{m_2}$,  defined by
$$
u(n)=\bigl(u_1(n)\! \! \!\pmod{1},\ldots,u_{m_2}(n)\! \! \!\pmod{1}\bigr)
$$
is equidistributed  on $\T^{m_2}$.

Then for $m_{\T^{m_1}}$-almost every $x\in \T^{m_1}$   the sequence
$(T^{n^{d}}x,u(n))_{n\in\N}$
is equidistributed on $\T^{m_1}\times \T^{m_2}$. Furthermore,  the set of full $m_{\T^{m_1}}$-measure  can be chosen to depend only on the transformation $T$
(so independently of the sequence $(u(n))_{n\in\N}$).
\end{lemma}
\begin{proof}
Suppose that  $T\colon \T^{m_1}\to\T^{m_1}$ is defined by $Tx=Sx+b$ for some unipotent homomorphism $S$ of $\T^{m_1}$ and $b\in \T^{m_1}$.
We claim that   the desired equidistribution
property  holds provided that  $x$ satisfies the following condition:
 \begin{equation}\label{E:condition}
  \text{ If }\ k_1\cdot \tilde{b}+k_2\cdot x=0 \bmod 1 \text{ for some }
k_1,k_2\in \Z^{m_1}, \text{ then } k_2={\bf 0 }
 \end{equation}
 where $\tilde{b}$ is defined in \eqref{E:new} below.
 This defines a set of full measure in $\T^{m_1}$ that depends only on the transformation $T$.

 Let  $x_0$ be any point in $\T^{m_1}$ that satisfies \eqref{E:condition}.
Let $\chi$ be a non-trivial character of $\T^{m_1}\times \T^{m_2}$.
Then $\chi=(\chi_1,\chi_2)$ for some characters  $\chi_1$  of $\T^{m_1}$ and
 $\chi_2$  of $\T^{m_2}$, and  at least one of  $\chi_1$ and $\chi_2$ is non-trivial. By Weyl's equidistribution theorem,
in order to verify that the sequence $(T^{n^{d}}x_0,u(n))_{n\in\N}$
is equidistributed on $\T^{m_1}\times \T^{m_2}$,
 it  suffices to show that
\begin{equation}\label{E:equi}
\lim_{N-M\to\infty} \frac{1}{N-M}\sum_{n=M}^{N-1}\chi_1(T^{n^{d}}x_0)\cdot \chi_2(u(n))=0.
\end{equation}
If $\chi_1=1$, then \eqref{E:equi} holds because, by assumption, the sequence $(u(n))_{n\in\N}$ is equidistributed on $\T^{m_2}$.
Suppose now that $\chi_1\neq 1$. Since  $S\colon \T^{m_1}\to \T^{m_1}$ is unipotent, we have  $(S-I)^{m_1}=0$. For
 $n\geq m_1$ a straightforward  inductive argument shows that, for
every $x\in\T^{m_1}$,
$$
T^nx=\sum_{k=0}^{m_1-1} \binom{n}{k}(S-I)^kx+\sum_{k=0}^{m_1-1}
 \binom{n}{k+1}(S-I)^kb\ .
$$
Therefore, the  sequence $(T^nx)_{n\in \N}$ is polynomial in $n$ and
\begin{equation}
\label{eq:Tnx}
T^nx=x+n\big(\tilde{S}x+\tilde{b}\big)+\text{higher order terms}
\end{equation}
where
\begin{equation}\label{E:new}
\tilde{S}=\sum_{k=1}^{m_1-1}\frac{(-1)^{k-1}}{k}(S-I)^k, \quad
\tilde{b}=\sum_{k=0}^{m_1-1}
 \frac{(-1)^k}{k+1}(S-I)^kb\ .
\end{equation}

\begin{claim*}
$\chi_1(\tilde{S}x_0+\tilde{b})=e(\alpha)$ for some irrational number $\alpha$.
\end{claim*}
Suppose on the contrary that
$\chi_1\big(\tilde{S}x_0+\tilde{b}\big)$ is rational.
After replacing $\chi_1$ by some power of $\chi_1$ we can assume that
$\chi_1\big(\tilde{S}x_0+\tilde{b}\big)=1$.
We write  $\chi_1(x)=e(k_1\cdot x)$, $x\in \T^{m_1}$, where $k_1$ is some non-zero
element of  $\Z^{m_1}$.
Then $k_1\cdot (\tilde{S}x_0+\tilde{b})=0\pmod{1}$, or equivalently
\begin{equation}\label{E:*}
k_1\cdot \tilde{b} +(k_1\cdot \tilde{S})\cdot x_0=0 \pmod{1}.
 \end{equation}
 Combining  \eqref{E:condition} and \eqref{E:*} we get that $k_1\cdot \tilde{S}=0 $. Using \eqref{E:*}  again we get that $k_1\cdot \tilde{b}=0\pmod{1}$. Hence, $\chi_1\circ \tilde{S}=1$
 and $\chi_1(\tilde{b})=1$.
Let  $d$ be the smallest positive integer such that  $\chi_1\circ (S-I)^d=1$ (such a $d$ exists since $S$ is unipotent).
If $d\geq 2$, then since  $\chi_1\circ \tilde{S}=1$ we get
$\chi_1\circ \tilde{S}\circ (S-I)^{d-2}=1$ and using the form of $\tilde{S}$ in \eqref{E:new}   we deduce that
$\chi_1\circ (S-I)^{d-1}=1$,  contradicting the minimality of $d$. Hence, $d=1$, that is,
  $\chi_1\circ (S-I)=1$.
Furthermore, since $\chi_1(\tilde{b})=1$ and $\chi_1\circ (S-I)=1$,
using the form of $\tilde{b}$ in \eqref{E:new} we deduce that $\chi_1(b)=1$.
  Therefore, $\chi_1\circ S=\chi_1$ and $\chi_1(b)=1$. Hence, $\chi_1\circ T=\chi_1$,
and since $\chi_1\neq 1$, this contradicts our assumption that the transformation $T$ is  ergodic. This completes the proof
of the claim.

From~\eqref{eq:Tnx} we conclude that
$\chi_1(T^nx_0)=e(c+n\alpha +n^2p(n))$
for some $c\in\R$, some irrational $\alpha$,
and  some polynomial $p\in \R[t]$.
Using this, and our assumption that  all the polynomials $u_i(t)$ are
divisible by  $t^{d+1}$, we get that
$$
\chi_1(T^{n^{d}}x_0)\cdot \chi_2(u(n))=e(c+n^{d}\alpha+\text{higher order terms}).
$$
Since $\alpha$ is irrational, it follows from this identity and  Weyl's equidistribution criterion
that~\eqref{E:equi} holds. This completes the proof.
\end{proof}

\begin{lemma}\label{L:NilEqui}
Let  $X=G/\Gamma$ be a connected nilmanifold,  $a\in G$ be an ergodic element,
and $d\in \N$. Let $Y=H/\Delta$ be a nilmanifold, $(g(n)y)_{n\in\N}$
defined by $g(n)=a_1^{p_1(n)}\cdot\ldots\cdot a_\ell^{p_\ell(n)}$
be a polynomial sequence on $Y$, and suppose that  the polynomials $p_1,\ldots,p_\ell$
 are all divisible by $t^{d+1}$.

Then there exists $r_0\in\N$ such that   for $m_X$-almost every $x\in X$ we have:
For every $r\in r_0\N$,
  the sequence $\big((a^{(rn)^{d}}x,g(rn)y)\big)_{n\in\N}$
 is equidistributed on
the set $X\times\cl_Y\{g(rn)y, n\in\N\}$.
 Furthermore,  the set of full $m_X$-measure can be chosen to depend
 only on the element $a\in G$ (so  independently of $Y$, $y$, and $g(n)$).
\end{lemma}
\begin{remark}
It is crucial for our subsequent applications that the full $m_X$-measure set of the lemma   does not depend
on the polynomial sequence  $(g(n)y)_{n\in\N}$. It is for this reason that we require the polynomials $p_1,\ldots, p_\ell$   to be divisible by $t^{d+1}$.
\end{remark}
\begin{proof}
\noindent \emph{The connected case.}
Suppose first that the set $\cl_Y\{g(n)y,  n\in\N\}$ is
connected.  In this case we are going to show that $r_0=1$ works.

First, by part~\eqref{it:TL1i} of Theorem~\ref{T:L1}, the set  $\cl_Y\{g(n)y, n\in\N\}$
 is  a sub-nilmanifold of $Y$. Substituting this set for $Y$ we can
assume that $Y=\cl_Y\{g(n)y, n\in\N\}$.
By part~\eqref{it:TL1i} of Theorem~\ref{T:L1},  we have
\begin{equation}\label{E:equidistributed}\emph{the sequence }
(g(n)y) \emph{ is equidistributed in } Y.
\end{equation}
\noindent\textbf{(a)} First we claim that it suffices to show
\begin{itemize}
\item[(i)] For $m_X$-almost every $x\in X$, where the set of
full measure depends only on $a$, the sequence
$\bigl((a^{n^{d}}x,g(n)y)\bigr)$ is equidistributed on the set $X\times
Y$.
\end{itemize}
Indeed, since the nilmanifolds $X$  and  $X\times Y$  are connected,
it follows by part (iii) of Theorem~\ref{T:L1}
that for every $r\in \N$ we have
$\cl_Y\{g(rn)y, n\in\N\}=Y$, and
 for every $r\in \N$ and every $x$ in the set defined in~(ii),
the sequence
$\bigl((a^{(rn)^{d}}x,g(rn)y)\bigr)$ is equidistributed on the set
$X\times Y$. This proves the claim.

\noindent\textbf{(b)}
Next, we use the convergence criterion given in  Theorem~\ref{T:L2}.

Let $A_X=G/([G_0,G_0]\Gamma)$ be the affine torus of $X$,  $A_Y =H/([H_0,H_0]\Delta)$ be
the affine torus of $Y$, and  $ \pi_{A_X}\colon X\to A_X$,  $\pi_{A_Y}\colon Y\to A_Y$
be the corresponding natural projections.
We first remark that $A_X\times A_Y$ is the affine torus of $X\times
Y$, with projection $\pi_X\times\pi_Y$.

Since the sequence $(g(n)y)$ is equidistributed in $Y$, the projection of
this sequence  onto $A_Y$ is equidistributed on $A_Y$.
   By Theorem~\ref{T:L2}, in order to show the required
equidistribution property~(i), it suffices
to verify the following
statement:
\begin{itemize}
\item[(ii)]
 For $m_X$-almost every $x\in X$, where the set of full measure depends only on  $a$, the sequence
 $$
 \big((a^{n^{d}} \pi_{A_X}(x), a_1^{p_1(n)}\cdot\ldots\cdot a_\ell^{p_\ell(n)}\pi_{A_Y}(y))\big)
 $$ is
equidistributed on $A_X\times A_Y$.
\end{itemize}
This statement is the same as~(i), with $A_X$ substituted for $X$,
$A_Y$ substituted for $Y$. We remark that all the hypotheses of the
lemma remain valid when we make this substitution.

Therefore, using the identification explained in
Section~\ref{SS:equidistribution}, we can restrict, without loss of
generality, to the case where  $X=\T^{m_1}$ for some $m_1\in\N$,
the translation $T_a\colon x\mapsto ax$ on $X$ is an ergodic
unipotent affine transformation of $\T^{m_1}$, and where
$Y=\T^{m_2}$ for some integer $m_2\in\N$ and for $i=1,\cdots,\ell$
the translation $T_{a_i}\colon y\mapsto a_iy$ on $Y$ is a unipotent
affine transformation of $\T^{m_2}$. Moreover,
by~\eqref{E:equidistributed}, the sequence
$(T_{a_1}^{p_1(n)}\cdot\ldots\cdot T_{a_\ell}^{p_\ell(n)}y)$ is
equidistributed on $\T^{m_2}$.

Since the uniform distribution is not affected by translation,
the  statement (ii)
can be rewritten in the following equivalent form:
\begin{itemize}
\item[(iii)]
 For $m_{\T^{m_1}}$-almost every $x\in \T^{m_1}$, where the set of full measure depends only on  the transformation $T_{a_1}$,  the sequence
 $$
 \big((T_{a}^{n^{d}} x, T_{a_1}^{p_1(n)}\cdot\ldots\cdot
T_{a_\ell}^{p_\ell(n)}y-y)\big)
 $$ is
equidistributed on $\T^{m_1}\times \T^{m_2}$.
\end{itemize}

\noindent\textbf{(c)}
Define the sequence $(u(n))_{n\in\N}$ with values in  $\T^{m_2}$ by
$$
u(n)=T_{a_1}^{p_1(n)}\cdot\ldots\cdot T_{a_\ell}^{p_\ell(n)}y - y\ .
$$
For $i=1,\ldots,  \ell$, since  $T_{a_i} $ is a unipotent affine
transformation, $T_{a_i}^ny$ is given for every $n$ by a formula
similar to~\eqref{eq:Tnx}. Therefore, for $j=1,\ldots, m_2$,  each coordinate $u_j(n)$  of
$u(n)$ is
a polynomial in $n$ with real coefficients and without a constant term. Moreover, since by
hypothesis the polynomials
$p_i(t)$ are divisible by $t^{d+1}$, all the polynomials $u_j(t)$ are
divisible by $t^{d+1}$.

 Hence,  Lemma~\ref{L:affine} is applicable and the statement (iii) is
proved.
  This completes the proof of the result
 in the case where the set $\cl_Y\{g(n)y, n\in\N\}$  is connected.

\subsubsection*{The general case}
Lastly we deal with the case where the set $\cl_Y\{g(n)y, n\in\N\}$  is not necessarily connected.
 By  Theorem~\ref{T:L1}, there exists
an $r_0\in\N$ such that the set $\cl_Y\{g(r_0n)y, n\in\N\}$ is connected.
Substituting the sequence $(g(r_0n)y)$ for  $(g(n)y)$
and $a^{r_0^d}$ for $a$ (which is again an ergodic element), the previous argument
shows the advertised result for this value of $r_0\in \N$.
 This completes the
proof of the result in the general case.
\end{proof}
We deduce from the previous lemma a  result that is more suitable for our purposes:
\begin{corollary}\label{C:gfd}
Let  $X=G/\Gamma$ be a  nilmanifold,   $a\in G$ be an ergodic element,
$f\in \CC(X)$ with $\E_{m_X}(f|\krat(T_a))=0$, and  $d,d_1\ldots,d_\ell\in \N$ with $d<d_i$ for $i=1,\ldots,\ell$.
Suppose that   $(u_{1,n})_{n\in\N},\ldots,(u_{\ell,n})_{n\in\N},$ are finite step nilsequences.

Then there exists $r_0\in\N$  such that, for $m_X$-almost every $x\in
X$,
  the following holds: For every $r\in r_0\N$ we have
\begin{equation}\label{E:pol}
\lim_{N-M\to \infty} \frac{1}{N-M}\sum_{n=M}^{N-1}
 f(a^{(rn)^{d}}x)\cdot u_{1,(rn)^{d_1}}\cdot \ldots\cdot u_{\ell,(rn)^{d_\ell}}=0.
\end{equation}
 Furthermore,  the set of full $m_X$-measure can be chosen to depend only on the element $a\in G$.
\end{corollary}
\begin{proof}
\emph{The connected case.} Suppose first that the nilmanifold $X$ is connected.
Using an approximation argument we can assume that for $i=1,\ldots,\ell$ the sequence $u_{i,n}$ is
a basic finite step nilsequence. In this case,
for $i=1,\ldots,\ell$ there exist nilmanifolds $X_i=G_i/\Gamma_i$,
elements $a_i\in G_i$, and functions $f_i\in \CC(X_i)$ such that
$u_{i,n}=f_i(a_i^n\Gamma_i)$.
We define $\tilde G=G_1\times\cdots\times G_\ell$, $\tilde \Gamma=
\Gamma_1\times\cdots\times\Gamma_\ell$, and
$\tilde X=X_1\times\cdots\times X_\ell=\tilde G/\tilde\Gamma$.
Let $(g(n))_{n\in\N}$ be the polynomial
sequence in $\tilde G$ given by
$g(n)=(a_1^{p_1(n)},\ldots,a_\ell^{p_\ell(n)})$ for every $n$.

Lemma~\ref{L:NilEqui} gives that there exists $r_0\in\N$ such that
for $m_X$-almost every $x\in X$
we have:
For every $r\in r_0\N$, the sequence
$(a^{(rn)^{d}}x,g(rn)\tilde{\Gamma})$ is
equidistributed on the nilmanifold $X\times Y$ where
$Y=\cl_{\tilde X}\{g(rn)\tilde{\Gamma},n\in\N\}$.

Therefore, for every
$f\in \CC(X)$ and $F\in \CC(\tilde{X})$  we have
$$
\lim_{N-M\to \infty} \frac{1}{N-M}\sum_{n=M}^{N-1}
f(a^{(rn)^{d}}x)\cdot F(g(rn)\tilde{\Gamma})=\int f(x) \ dm_X(x)\cdot \int F(\tilde{x}) \ dm_{\tilde{X}}(\tilde{x}).
$$
Letting $F=f_1\cdot \ldots \cdot f_\ell$, and using that $\int f\, dm_X=0$,   we get the advertised identity. This completes the proof in the case where the nilmanifold $X$ is connected.
\medskip

\noindent \emph{The general case.} Let $X_0$ be the connected component of the nilmanifold $X$. Since $a$ is an ergodic element,
 there exists
$k\in \N$ such that the nilmanifold $X$ is the disjoint union of the sub-nilmanifolds
$X_i=a^iX_0$, $i=1,\ldots,k$, and $a^{k}$ acts ergodically on each $X_i$.
Furthermore, since $\E_{m_X}(f|\krat(T_a))=0$, we have  $\int f\, dm_{X_i}=0$ for $i=1,\ldots,k$.
For $i=1,\ldots,k$, we can apply the previously established ``connected  result'', for the translation $a^{k^{d}}$ in place of $a$, acting (ergodically) on the connected sub-nilmanifolds $X_i$, and the
 nilsequences $(u_{j,k^{d_j}n})$ in place of $(u_{j,n})$, $j=1,\ldots,\ell$.  We get that there exist $r_i\in \N$ such that for every $r\in kr_i\N$ equation  \eqref{E:pol}
holds for $m_{X_i}$-almost every $x\in X_i$. It follows that  if  $r_0=k\prod_{i=1}^k r_i$, then
for every $r\in r_0\N$ equation  \eqref{E:pol}
holds for $m_{X}$-almost every $x\in X$. This completes the proof in the general case.
\end{proof}

\subsection{Proof of the convergence result (Proposition~\ref{P:OptimalChar})}
In this section we prove Proposition~\ref{P:OptimalChar} by induction on the number of transformations
involved. The key ingredient in the proof of the inductive step is the following   special case of  Proposition~\ref{P:OptimalChar}:

\begin{lemma}\label{P:OptimalChar'}
Let $(X,\X,\mu,T_1,\cdots,T_\ell)$ be a system. Let
$d_1,\ldots,d_\ell\in \N$ be distinct and suppose that $d_1<d_i$ for
$i=2,\ldots,\ell$. Let $f_1,\ldots,f_\ell\in L^\infty(\mu)$ and
suppose  that $f_1\perp\krat(T_1)$.

Then  for every $\varepsilon>0$, there exists $r_0\in \N$, such that  for every $r\in r_0\N$, we have
$$
\lim_{N-M\to\infty}\norm{\frac{1}{N-M}\sum_{n=M}^{N-1} f_1(T_1^{(rn)^{d_1}}x)\cdot \ldots\cdot f_\ell(T_\ell^{(rn)^{d_\ell}}x)}_{L^2(\mu)}
\leq \varepsilon.
$$
\end{lemma}
\begin{proof}
Let $\varepsilon>0$. Without loss of generality we can assume that
all the functions involved are bounded by $1$. From
Theorem~\ref{T:CharA} we have that there exists $k\in\N$, depending
only on $\max(d_1,\cdots,d_\ell)$ and $\ell$, such that if
$f_i\perp \cZ_{k,T_i}$ for some $i=1,\ldots,\ell$, then the
corresponding multiple ergodic averages converge to $0$ in
$L^2(\mu)$.
 Therefore, we can assume that
$f_i\in L^\infty(\cZ_{k,T_i}, \mu)$ for $i=1,\ldots,\ell$.
Then  Proposition~\ref{L:ApprNil} shows  that for $i=2,\ldots,\ell$
  there exist  functions $\tilde{f}_i$, with $L^\infty$-norm  bounded by $1$,  that satisfy
\begin{enumerate}
\item
 $\tilde{f}_i\in L^\infty(\cZ_{k,T_i},\mu)$ and $\norm{f_i-\tilde{f}_i}_{L^2(\mu)}\leq \varepsilon/(2\ell+2)$\ ;
\item
  for every $r\in \N$ and  $x\in X$ the sequence $(\tilde{f}_i(T_i^{n^{d_i}}x))_{n\in\N}$ is a $(d_ik)$-step nilsequence.
\end{enumerate}

An easy computation then shows that  in order to prove the announced claim, it suffices to show the following:
 If $f_1\in L^\infty(\cZ_{k,T_1},\mu)$  and $f_1\bot \krat(T_1)$, then there exists $r_0\in \N$ such that
for every $r\in r_0\N$  we have
\begin{equation}\label{E:NilAv'}
\lim_{N-M\to\infty}\norm{\frac{1}{N-M}\sum_{n=M}^{N-1} f_1(T_1^{(rn)^{d_1}}x)\cdot \tilde{f}_2(T_2^{(rn)^{d_2}}x)\cdot \ldots\cdot
\tilde{f}_\ell(T_\ell^{(rn)^{d_\ell}}x)}_{L^2(\mu)} \leq \frac{\varepsilon}{2}.
\end{equation}
 \medskip

\noindent \emph{The ergodic case.} Suppose first that the transformation $T_1$ is ergodic.
Since $f_1\in L^\infty(\cZ_{k,T_1},\mu)$, after using an appropriate
conjugation we can assume that $T_1$ is an inverse limit of
nilsystems.
Furthermore, after using an approximation argument we can
assume that $T_1=T_a$ where $a$ is an ergodic rotation on a
nilmanifold $X=G/\Gamma$, and $f_1\in C(X)$, while still maintaining
our assumption that $f_1\perp\krat(T_1)$.
(If $f\bot \cD$  where $\cD$ is any sub-$\sigma$-algebra of $\X$, and $g$ is such that $\norm{f-g}_{L^1(\mu)}\leq \varepsilon/2$,
 then $\norm{\E(g|\cD)}_{L^1(\mu)}\leq \varepsilon/2$. Therefore, $\norm{f-\tilde{g}}_{L^1(\mu)}\leq \varepsilon$ where
 $\tilde{g}=g-\E(g|\cD)$, and $\E(\tilde{g}|\cD)=0$.)
In this case,  combining property (ii)  above and Corollary~\ref{C:gfd}, we get  that there exists $r_0\in\N$
such that for every $r\in r_0\N$ the averages \eqref{E:NilAv'} converge to $0$ for $m_X$-almost every  $x\in X$,
 and as a result in $L^2(m_X)$.
 This completes the proof of  \eqref{E:NilAv'} in the case where the transformation $T_1$ is ergodic.
\medskip

\noindent \emph{The  general case.}   Suppose now that the transformation  $T_1$ is not necessarily ergodic. Let $\mu=\int \mu_x\ d\mu$ be the ergodic decomposition of $\mu$ with respect to
  the transformation $T_1$. Since $f_1\in
L^\infty(\cZ_{k,T_1,\mu}, \mu)$,   Corollary~\ref{cor:Zkmux} shows that
for $\mu$-almost every $x\in X$ we have
$f_1\in L^\infty(\cZ_{k,T_1,\mu_x}, \mu_x)$. Furthermore, since $\E_\mu(f_1|\krat(T,\mu))=0$, we have  for $\mu$-almost every
$x\in X$ that  $\E_{\mu_x}(f_1|\krat(T,\mu_x))=0$.

   For every $r_0\in \N$  we define the $\mu$-measurable set
 \begin{multline*}
X_{r_0}= \bigl\{x\in X\colon   \eqref{E:NilAv'}  \text{ holds
 for every } r\in r_0\N,
   \text{ with }
 \mu_x  \text{ in place of } \mu,
\text{ and } \varepsilon/2 \text{ in place of } \varepsilon  \bigr\}.
 \end{multline*}
Notice that when we previously established  the ``ergodic case'', we
did not use the invariance of the measure $\mu$ under the transformations $T_i$ for $i\neq  1$; we merely used the fact that for  $i=2,\ldots,\ell$,  for $\mu$-almost every $x\in X$, the sequences $(\tilde{f}_i(T_i^{n}x))_{n\in\N}$ are  $k$-step nilsequences.
Hence,  we can use    the previously established
 ``ergodic  result'' for $\mu$-almost every measure $\mu_x$, and   conclude that
   $$
   \mu(\bigcup_{r_0\in \N}X_{r_0})=1.
   $$

Also, we clearly have $X_{r}\subset X_s$ if $r$ divides $s$.
It follows that there exists $r_0\in \N$ such that
$$
  \mu(X_{r_0})\geq 1-\varepsilon/4.
$$
As a direct consequence, for this choice of $r_0$,
equation~\eqref{E:NilAv'} holds for every $r\in r_0\N$. This completes the proof.
\end{proof}

We are now ready to prove Proposition~\ref{P:OptimalChar}.
\begin{proof}[Proof of Proposition~\ref{P:OptimalChar}]
Without loss of generality we can assume that $d_1<d_i<d_\ell$ for $i=2,\ldots,\ell-1$ and $\norm{f_i}_{L^\infty(\mu)}\leq 1$ for $i=1,\ldots,\ell$.

We are going to use induction on the number of transformations $\ell$. For $\ell=1$ the statement is known (Chapter 3 in \cite{Fu81}) and in fact it holds with $r_0=1$ and $\varepsilon=0$. Suppose that $\ell\geq 2$, and the statement holds for $\ell-1$ transformations. We are going to show that it holds for $\ell$ transformations. Namely, we are going to show that if $f_i\perp\krat(T_i)$ for some $i\in \{1,\ldots,\ell\}$, then   for every $\varepsilon>0$, there exists $r_0\in \N$, such that  for every $r\in r_0\N$ we have
$$
\lim_{N-M\to\infty}\norm{\frac{1}{N-M}\sum_{n=M}^{N-1} f_1(T_1^{(rn)^{d_1}}x)\cdot \ldots\cdot f_\ell(T_\ell^{(rn)^{d_\ell}}x)}_{L^2(\mu)}
\leq \varepsilon.
$$

Let $\varepsilon>0$.
If $f_1\perp\krat(T_1)$,  then the result follows from Lemma~\ref{P:OptimalChar'}.
So we can assume that  $f_i\perp\krat(T_i)$ for some $i\in \{2,\ldots,\ell\}$.
By Lemma~\ref{P:OptimalChar'} we can assume that the function $f_1$ is $\krat(T_1)$-measurable.
Furthermore, using a standard approximation argument we can  assume that the function $f_1$  is $\K_{r_1}(T_1)$-measurable
for some $r_1\in \N$.
Since
 for every $r\in r_1 \N$ we have $T^{r}f_1=f_1$,  it remains to find $r_2\in r_1\N$, such  that  for every $r\in r_2\N$  we have
$$
\lim_{N-M\to\infty}\norm{\frac{1}{N-M}\sum_{n=M}^{N-1} f_2(T_2^{(rn)^{d_1}}x)\cdot \ldots\cdot f_\ell(T_\ell^{(rn)^{d_\ell}}x)}_{L^2(\mu)}
\leq \varepsilon.
$$
Such an integer $r_2$ exists   from the induction hypothesis. This completes the induction and the proof.
\end{proof}

\appendix
\section{Some ``simple'' proofs of  special cases of the main results}
\label{sec:appendix}

It turns out that Theorem~\ref{T:CharA} can be strengthened, and the
proof of Theorems~\ref{T:Conv}, \ref{T:CharA}, and
\ref{T:LowerBounds}, can be greatly simplified in some
interesting special cases, namely when $\ell=2$ and one of the two
polynomials is linear.
Such a simplification is feasible because of
the nature of the averages involved;
it turns out  to be possible to  get simple characteristic factors
by using a variation of van der Corput's Lemma, and then appealing to a known
result from \cite{FK1}.
 We take the opportunity in this section to give these
simple arguments. Hopefully, the non-persistent reader, that does not
want to embark to the details of the more complicated proofs of our
main results, will benefit from the proofs of the special cases given
here.

The key ingredient in the proofs is the following result:
\begin{theorem}[\cite{FK1}]\label{T:FrKr}
Let $(X,\X,\mu,T)$ be a system and
suppose that the integer polynomials $1,p,q$ are linearly independent.
Let $f,g\in L^\infty(\mu)$ and suppose
that $f\perp \krat(T)$ or $g\perp\krat(T)$.

Then
$$
\lim_{N-M\to\infty}\frac{1}{N-M}\sum_{n=M}^{N-1} f(T^{p(n)}x)\cdot g(T^{q(n)}x)=0
$$
where the convergence takes place in $L^2(\mu)$.
\end{theorem}
\begin{remark}
The proof in \cite{FK1} is given for ergodic systems, but the announced result follows
directly from this, since $f\bot \krat(T,\mu)$ implies that $f\bot \krat(T,\mu_x)=0$
for $\mu$-almost every $x\in X$, where as usual,
 $\mu=\int \mu_x\ d\mu(x)$ is the ergodic decomposition of $\mu$.
\end{remark}

We are also going to use the following   variation of the classical elementary lemma of van der Corput.
Its proof is  a straightforward modification of  the one given
in~\cite{Be}.

\begin{lemma}\label{L:N-VDC'}
Let  $\{v_{N,n}\}_{N,n\in\N}$ be a bounded  sequence of vectors in a
Hilbert space.
 For every $h\in \N$ we set
$$
b_h=\overline{\lim}_{N\to \infty}\Big|\frac{1}{N}
\sum_{n=1}^{N}<v_{N,n+h},v_{N,n}>\Big|.
$$
Suppose that
$$
\lim_{H\to\infty}\frac{1}{H}\sum_{h=1}^H b_h=0.
$$

Then
$$
\lim_{N\to\infty}
\norm{\frac{1}{N}\sum_{n=1}^{N} v_{N,n}}=0.
$$
\end{lemma}

  We start with the following strengthening of Theorem~\ref{T:CharA} in our particular setup:
\begin{theorem}\label{T:1}
Let $(X,\X,\mu,T,S)$ be a system.
Let  $f,g\in L^\infty(\mu)$ and suppose that  either $f\perp\krat(T)$ or
$g\perp\krat(S)$.

Then for every polynomial $p\in\Z[t]$ with $\deg (p)\geq 2$ we have
\begin{equation}\label{E:vbn}
\lim_{N-M\to\infty}\frac{1}{N-M}\sum_{n=M}^{N-1} f(T^nx)\cdot g(S^{p(n)}x)=0
\end{equation}
where the convergence takes place in $L^2(\mu)$.
\end{theorem}
\begin{proof}
Suppose first that  $\E(g|\krat(S))=0$. It suffices to show that for every  sequence of intervals
$(I_N)_{N\in\N}$ with length increasing to infinity, the averages in $n$ over the intervals $I_N$ of
$$
\int h_N(x) \cdot f(T^nx) \cdot g(S^{p(n)}x) \ d\mu
$$
converge to $0$, where $h_N(x)=\frac{1}{|I_N|}\sum_{n\in I_N}f(T^nx)\cdot g(S^{p(n)}x)$.
Equivalently, it suffices to show that the averages over the intervals $I_N$ of
$$
\int f(x) \cdot h_N(T^{-n}x) \cdot g( S^{p(n)} T^{-n}x) \ d\mu
$$
converge to  $0$.
Using the Cauchy-Schwarz inequality it suffices to show that the averages over the intervals $I_N$ of
$$
h_N(T^{-n}x) \cdot g(S^{p(n)}T^{-n}x)
$$
converge to $0$ in $L^2(\mu)$.
By Lemma~\ref{L:N-VDC'}   it suffices to show that for every $m\in \N$ the averages in $n$ over the intervals $I_N$ of
$$
\int {h_N}(T^{-n}x) \cdot {g}( S^{p(n)} T^{-n}x)
\cdot h_N(T^{-(n+m)}x)\cdot g(S^{p(n+m)}T^{-(n+m)}x) \ d\mu
$$
converge to $0$.
We compose with  the transformation $T^{n}$ and use the Cauchy-Schwarz inequality. It
 suffices to show that for every $m\in\N$  the averages in $n$ over the intervals $I_N$ of
$$
 g(S^{p(n)}x)
\cdot g(T^{-m}S^{p(n+m)}x)
$$
converge to $0$ in $L^2(\mu)$.
Since $\deg (p)\geq 2$, for every $m\in\N$ the polynomials  $1, p(n), p(n+m)$  are linearly independent.
 Since $g\perp\krat(S)$,
  Theorem~\ref{T:FrKr}  verifies that the last identity holds.

It remains to show that if  $f\perp\krat(T)$, then the averages over the intervals $I_N$ of
 $$
f(T^nx)\cdot g(S^{p(n)}x)
$$
converge to $0$ in $L^2(\mu)$.
 Using the previously established property we get that the above limit remains unchanged if we replace
  the function  $g$ with  the function $\E(g|\krat(S))$. Furthermore,  using an approximation argument and linearity,
 we can assume that $Sg=e(r)g$ for some $r\in \QQ$. In this case, it suffices to show that the averages over the intervals $I_N$ of
$$
  f(T^nx)\cdot e(rp(n))
$$
converge to $0$ in $L^2(\mu)$. Using the spectral theorem for unitary operators it
suffices to show that for every $r\in \QQ$ we have
\begin{equation}\label{E:spectral}
\lim_{N\to\infty}\Big{|}\Big{|}\frac{1}{|I_N|}\sum_{n\in I_N}   e(nt+rn^2)\Big{|}\Big{|}_{L^2(\sigma_f(t))}= 0
\end{equation}
where $\sigma_f$ denotes the spectral measure of the function $f$.
 Since $f\bot \krat (T)$, the measure $\sigma_f$ has no rational point masses.
 Furthermore, as is well known, for $t$ irrational the averages in \eqref{E:spectral} converge to $0$ pointwise.
Combining these two facts, and using the bounded convergence theorem, we deduce that \eqref{E:spectral} holds.
This completes the proof.
\end{proof}
We deduce the following special case of Theorem~\ref{T:Conv}:
\begin{theorem}\label{T:2}
Let $(X,\X,\mu,T,S)$ be a system and $f,g\in L^\infty(\mu)$. Let  $p\in\Z[t]$ with $\deg (p)\geq 2$.

Then  the limit
$$
\lim_{N-M\to\infty}\frac{1}{N-M}\sum_{n=M}^{N-1} f(T^nx)\cdot g(S^{p(n)}x)
$$
exists in $L^2(\mu)$.
\end{theorem}
\begin{proof}
By Theorem~\ref{T:1} we can assume that the function $f$ is  $\krat(T)$-measurable
and the function $g$ is $\krat(S)$-measurable.
Furthermore using an approximation argument we can assume that $T^rf=f$ and $T^rg=g$ for some $r\in \N$.
In this case the result is obvious.
\end{proof}
As a corollary we get an short proof for weak convergence of some multiple ergodic averages    recently
 studied by T.~Austin in~\cite{A2} (where strong convergence was proven when $p(n)=n^2$).
\begin{corollary}
Let $(X,\X,\mu,T,S)$ be a system
 and $f,g\in L^\infty(\mu)$.
Let  $p\in\Z[t]$ with $\deg (p)\geq 2$.

Then  the averages
\begin{equation}\label{E:fvb}
\frac{1}{N-M}\sum_{n=M}^{N-1}  f(T^{p(n)}x)\cdot g(T^{p(n)}S^nx)
\end{equation}
converge weakly in $L^2(\mu)$ as $N-M\to\infty$. Furthermore, the limit is $0$   if either  $g\bot \krat(S)$
or $f\bot (\krat(T)\vee \krat(S))$.
\end{corollary}
\begin{proof}
Notice that for every $h\in L^\infty(\mu)$ the averages
of
$$
\int h(x) \cdot  f(T^{p(n)}x)\cdot g(T^{p(n)}S^nx)\ d\mu
$$
are equal to the averages of
\begin{equation}\label{E:tyr}
\int f(x)\cdot h(T^{-p(n)}x)\cdot g(S^nx) \ d\mu.
\end{equation}
 Theorem~\ref{T:2}  shows that the averages of \eqref{E:tyr}   converge, therefore
the averages  \eqref{E:fvb} converge weakly. Furthermore,
Theorem~\ref{T:1}
 shows that the averages of \eqref{E:tyr} converge to $0$ if either  $g\bot \krat(S)$
 or $h\bot \krat(T)$, and as a consequence they converge to $0$  if $f\bot (\krat(T)\vee \krat(S))$.
Therefore, if $g\bot \krat(S)$ or $f\bot (\krat(T)\vee \krat(S))$, then
 the averages of \eqref{E:tyr}  converge weakly to $0$.
 This completes the proof.
\end{proof}
Finally we establish the following result:
\begin{theorem}\label{T:LowerPoly}
Let $(X,\X,\mu, T,S)$ be a system and
 $A\in \X$. Let $p\in \Z[t]$ with $\deg (p)\geq 2$ and $p(0)=0$.

Then for every positive integer $k\geq 2$ and  $\varepsilon>0$  the set
$$
\{n\in\N\colon \mu(A\cap T^{-n}A\cap S^{-p(n)}A)> \mu(A)^3-\varepsilon\}
$$
has bounded gaps.
\end{theorem}
\begin{proof}
Let $\varepsilon>0$. There exists $r\in \N$ such that
\begin{align}\label{E:approxi}
\norm{\E({\bf 1}_A|\mathcal{K}_r(T))-\E({\bf 1}_A|\krat(T))}_{L^2(\mu)}
&\leq \varepsilon/3, \quad
\norm{\E({\bf 1}_A|\mathcal{K}_r(S))-\E({\bf 1}_A|\krat(S))}_{L^2(\mu)}
& \leq \varepsilon/3.
\end{align}
 It suffices to show  that
$$
\lim_{N-M\to\infty}\frac{1}{N-M}\sum_{n=M}^{N-1}\mu(A\cap T^{-rn}A\cap S^{-p(rn)}A)\geq \mu(A)^3-\varepsilon.
$$
Using   a straightforward modification  of Theorem~\ref{T:1}, where $T^n$ is replaced with $T^{rn}$, we see that
 the previous limit is equal to the limit of the averages of
$$
\int {\bf 1}_A\cdot T^{-rn}\E({\bf 1}_A|\krat(T))\cdot
S^{-p(rn)}\E({\bf 1}_A|\krat(S))\ d\mu.
$$
Using  \eqref{E:approxi} we see that the last limit is $\varepsilon$-close to the limit of the averages of
$$
\int {\bf 1}_A\cdot T^{-rn}\E({\bf 1}_A|\mathcal{K}_r(T))\cdot S^{-p(rn)}
\E({\bf 1}_A|\mathcal{K}_r(S))\ d\mu.
$$
Since $T^rf=f$ for $\mathcal{K}_r(T)$-measurable functions $f$, $S^rf=f$
for $\mathcal{K}_r(S)$-measurable functions $f$, and $r|p(rn)$ for every $n\in\N$ (since $p(0)=0$),
the last limit is equal to
$$
\int {\bf 1}_A\cdot \E({\bf 1}_A|\mathcal{K}_r(T))\cdot \E({\bf
1}_A|\mathcal{K}_r(S)) \ d\mu.
$$
By Lemma~\ref{L:LowerBound}, the last integral is greater or equal than
$\mu(A)^3$, completing the proof.
\end{proof}

\end{document}